\documentclass{article}
\usepackage{latexsym}
\usepackage{amsmath}
\usepackage{amssymb}
\usepackage{graphicx}
\usepackage{subfigure}
\usepackage{ae}
\usepackage{multirow}

\usepackage{xypic}
\usepackage[all]{xy}

\usepackage{enumerate}

\newcommand{\lfis}{{\bf LFI}s}
\newcommand{\lfi}{{\bf LFI}}

\newcommand{\mbc}{{\bf mbC}}
\newcommand{\cplp}{\text{\bf CPL$^+$}}
\newcommand{\cpl}{\text{\bf CPL}}
\newcommand{\ecpl}{\ensuremath{\mathbf{CPL}^+_{\mathbf{e}}}}
\newcommand{\eecpl}{\ensuremath{\mathbf{CPL}_{\mathbf{e}}}}

\newcommand{\mbcciw}{{\bf mbCciw}}
\newcommand{\mbcci}{{\bf mbCci}}
\newcommand{\ci}{{\bf Ci}}

\newcommand{\ciore}{{\bf Ciore}}
\newcommand{\dacdot}{{\bf J3}}

\newcommand{\lfium}{{\bf LFI1}}
\newcommand{\mpt}{{\bf MPT}}
\newcommand{\luka}{\ensuremath{\mathbf{L3}}}
\newcommand{\nel}{\text{\bf N4}}

\newcommand{\malg}{{\bf MAlg}}
\newcommand{\swmbc}{\ensuremath{\mathbf{SW}_{\mbc}}}
\newcommand{\sw}{\ensuremath{\mathbf{SW}_{\ecpl}}}
\newcommand{\swmbcciw}{\ensuremath{\mathbf{SW}_{\mbcciw}}}
\newcommand{\sweecpl}{\ensuremath{\mathbf{SW}_{\eecpl}}}
\newcommand{\swmbcci}{\ensuremath{\mathbf{SW}_{\mbcci}}}

\newcommand{\swlfium}{\ensuremath{\mathbf{SW}_{\lfium_\cons}}}
\newcommand{\swciore}{\ensuremath{\mathbf{SW}_{\ciore}}}
\newcommand{\balg}{{\bf BAlg}}

\newcommand{\mA}{\ensuremath{\mathcal{A}}}
\newcommand{\mB}{\ensuremath{\mathcal{B}}}
\newcommand{\matM}{\ensuremath{\mathcal{M}}}
\newcommand{\Ksw}{\ensuremath{\mathbb{K}_{\ecpl}}}
\newcommand{\Kmbc}{\ensuremath{\mathbb{K}_{\mbc}}}
\newcommand{\Kmbcciw}{\ensuremath{\mathbb{K}_{\mbcciw}}}
\newcommand{\Kmbcci}{\ensuremath{\mathbb{K}_{\mbcci}}}
\newcommand{\Kci}{\ensuremath{\mathbb{K}_{\ci}}}

\newcommand{\Klfium}{\ensuremath{\mathbb{K}_{\lfium_\cons}}}

\newcommand{\Keecpl}{\ensuremath{\mathbb{K}_{\eecpl}}}
\newcommand{\Kciore}{\ensuremath{\mathbb{K}_{\ciore}}}

\newcommand{\axwci}{{\bf ciw}}
\newcommand{\axci}{{\bf ci}}

\newcommand{\axcons}{{\bf cons}}
\newcommand{\MP}{\textbf{MP}}
\newcommand{\axnegou}{\textbf{neg$\lor$}}
\newcommand{\axnege}{\textbf{neg$\land$}}
\newcommand{\axnegimp}{\textbf{neg$\imp$}}

\newcommand{\kax}{\textbf{Ax1}}
\newcommand{\axTrans}{\textbf{Ax2}}
\newcommand{\axed}{\textbf{Ax3}}
\newcommand{\axeea}{\textbf{Ax4}}
\newcommand{\axeeb}{\textbf{Ax5}}
\newcommand{\axouda}{\textbf{Ax6}}
\newcommand{\axoudb}{\textbf{Ax7}}
\newcommand{\axoue}{\textbf{Ax8}}
\newcommand{\axouimp}{\textbf{Ax9}}
\newcommand{\axtnd}{\textbf{Ax10}}
\newcommand{\axexp}{\textbf{bc1}}

\newcommand{\axce}{\textbf{ce}}
\newcommand{\axcf}{\textbf{cf}}
\newcommand{\axco}{\textbf{co}}

\newcommand{\algcpl}{classical implicative lattice}

\newcommand{\defin}{\ensuremath{~\stackrel{\text{{\tiny def }}}{=}~}}

\newcommand{\imp}{\rightarrow}
\newcommand{\cons}{\ensuremath{{\circ}}}

\newcommand{\valou}{\textit{vOr}}
\newcommand{\vale}{\textit{vAnd}}
\newcommand{\valimp}{\textit{vImp}}
\newcommand{\valnot}{\textit{vNeg}}
\newcommand{\valbola}{\textit{vCon}}

\newcommand{\valCi}{\textit{vCi}}

\newcommand{\valcef}{\textit{vCeCf}}

\newcommand{\valco}{\textit{vCo}}
\newcommand{\valcip}{\textit{vCIp}}
\newcommand{\valdm}{\textit{vDM}}

\newtheorem{definition}{\vspace{1mm}Definition}[section]
\newtheorem{remark}[definition]{\vspace{1mm}Remark}

\newtheorem{theorem}[definition]{\vspace{1mm}Theorem}
\newtheorem{corollary}[definition]{\vspace{1mm}Corollary}
\newtheorem{proposition}[definition]{\vspace{1mm}Proposition}
\newtheorem{Notation}[definition]{\sc Notation}

\newenvironment{proof}{\begin{trivlist}\item{\bf Proof:}}{\qurtains\end{trivlist}}

\def\pushright#1{{\parfillskip=0pt\widowpenalty=10000
  \displaywidowpenalty=10000\finalhyphendemerits=0
  \leavevmode\unskip\nobreak\hfil\penalty50\hskip.2em\null
  \hfill{#1}\par}}

\newcommand{\qurtains}     {\pushright{$\blacksquare$}\par \vspace{1ex}}

\begin{document}

\title{Non-deterministic algebraization of logics by  swap structures}

\author{Marcelo E.~Coniglio$^{1}$, Aldo Figallo-Orellano$^{2}$ and Ana C. Golzio$^{3}$\\ [2mm] %
{\small $^1$Institute of Philosophy and the Humanities (IFCH) and}\\
{\small Centre for Logic, Epistemology and The History of Science (CLE),}\\
{\small University of Campinas (UNICAMP), Campinas, SP, Brazil.}\\
{\small E-mail:~{coniglio@cle.unicamp.br}}\\[1mm]
{\small $^2$Department of Mathematics, National University of the South (UNS),}\\
{\small  Bahia Blanca, Argentina and}\\
{\small Centre for Logic, Epistemology and The History of Science (CLE),}\\
{\small University of Campinas (UNICAMP), Campinas, SP, Brazil.}\\
{\small E-mail:~{aldofigallo@gmail.com}}\\[1mm]
{\small $^3$Centre for Logic, Epistemology and The History of Science (CLE),}\\
{\small University of Campinas (UNICAMP), Campinas, SP, Brazil.}\\
{\small E-mail:~{anaclaudiagolzio@yahoo.com.br}}}

\date{}

\maketitle

\begin{abstract}
Multialgebras (or hyperalgebras, or non-deterministic algebras) have been very much studied in Mathematics and in Computer Science. In 2016 Carnielli and Coniglio introduced a class of multialgebras  called swap structures, as a  semantic framework for dealing with several logics of formal inconsistency (or \lfis) which cannot be semantically characterized by a single finite matrix.  In particular, these \lfis\ are not algebraizable by the standard tools of abstract algebraic logic.
In this paper,  the first steps towards a theory of non-deterministic algebraization of logics by swap structures are given. Specifically, a formal study of swap structures for \lfis\ is developed, by adapting concepts of universal algebra to multialgebras in a suitable way. A decomposition theorem similar to Birkhoff's representation theorem is obtained for each class of swap structures. Moreover, when applied to the 3-valued algebraizable logic \dacdot\ the usual class of algebraic models is recovered, and the swap structures semantics became twist-structures semantics (as introduced by Fidel-Vakarelov). This fact, together with the existence of a functor from the category of Boolean algebras to the category of swap structures for each \lfi, which is closely connected with Kalman's functor, suggests that swap structures can be considered as non-deterministic twist structures, opening so interesting possibilities for dealing with non-algebraizable logics by means of multialgebraic semantics. 
\end{abstract}

\


\section{Introduction} \label{intro}

As it is well-known, several logics in the hierarchy of the so-called {\em Logics of Formal Inconsistency} (in short \lfis, see \cite{CM,CCM,CC16}) cannot be semantically characterized by a single finite matrix. Moreover, they lie outside the scope of the usual techniques of algebraization of logics such as Blok and Pigozzi's method (see~\cite{blok:pig:89}). Several alternative semantical tools were introduced in the literature in order to deal with such systems: non-truth-functional bivaluations, possible-translations semantics, and non-deterministic matrices (or Nmatrices), obtaining so decision procedures for these logics.
However, the problem of finding an algebraic counterpart for this kind of logic, in a sense to be determined, remains open.

A semantics based on an special kind of multialgebra  called  {\em swap structure}  was proposed in \cite[Chapter~6]{CC16}, which generalizes the characterization results of \lfis\ by means of finite Nmatrices due to Avron (see~\cite{avr:05}). Moreover, the swap structures semantics allows soundness and completeness theorems by means of a very natural generalization of the well-known  Lindenbaum-Tarski process (for an example applied to non-normal modal logics see~\cite{CG16} and~\cite[Chapter~3]{GC17}).

Multialgebras (also known as hyperalgebras or non-deterministic algebras) have been very much studied in the literature. Besides their use in Logic by means of Nmatrices, they have been applied to several areas of Computer Science such as automata theory. Multialgebras has also  been studied in Mathematics, in areas such as algebra, geometry, topology, graph theory and probability theory. An historical survey on multialgebras can be found in~\cite[Chapter~1]{GC17}.

From the algebraic perspective, the formal study of multialgebras is not so immediate: the generalization from universal algebra to multialgebras of even basic conceps such as homomorphism, subalgebras and congruences is far to be obvious, and several different alternatives were proposed in the literature. In particular, the possibility of defining an algebraic theory of non-deterministic structures for logics along the same lines of the so-called abstract algebraic logic (see, for instance, \cite{font:16}) is an open question which deserves to be investigated.

This paper give some steps along this direction, by adapting concepts of universal algebra to multialgebras in a suitable way in order to analyze  categories of swap structures for some \lfis. Specifically, we will concentrate our efforts on the algebraic theory  of  the class $\Kmbc$ of swap structures for the logic \mbc\ (the weakest  system in the hierarchy of \lfis\ proposed in~\cite{CCM} and~\cite{CC16}).
In order to do this,  and taking into account that swap structures are special cases of multialgebras, a category of multialgebras over a given signature is considered, based on  very natural notions of homomorphism and submultialgebras. From this, products  and congruences are analyzed, showing that the class $\Kmbc$ is closed under substructures and products, but it is not closed under homomorphic images. From this,  it is possible to give a representation theorem for $\Kmbc$ (see Theorem~\ref{teo-repr-Kmbc}) which resembles  the well-known representation theorem for algebras obtained by G. Birkhoff in 1944 (see~\cite{GB2}). As a consequence of our result,  the class $\Kmbc$ is generated by the structure  with five elements, which is constructed over the 2-element Boolean algebra. Such structure is precisely Avron's 5-valued characteristic Nmatrix for \mbc\ introduced in~\cite{avr:05}.

This approach is extended to several axiomatic extensions of \mbc, including the 3-valued paraconsistent logic \dacdot\ (see~\cite{dot:dac:70}), which is algebraizable. The classes of swap structures for  each of such systems are subclasses of $\Kmbc$. They are obtained by requiring that its elements satisfy precisely the additional axioms which define the corresponding logic. Analogous Birkhoff-like representation theorems for each class of swap structures are found.  This allow a modular treatment of the algebraic theory of swap structures, as happens in the traditional algebraic setting.

In the case of the algebraizable 3-valued logic \dacdot, our representation theorem coincides with the original Birkhoff's representation theorem. Moreover, the swap structures became twist structures in the sense of Fidel~\cite{fid:78} and Vakarelov~\cite{vaka:77}. This fact, together with the existence of a functor from the category of Boolean algebras to the category of swap structures for each \lfi, which is closely connected with the Kalman's functor naturally associated to twist structures (see~\cite{kal:58,cin:86}), suggests that swap structures can be considered as non-deterministic twist structures, as analyzed in Section~\ref{secLFI1}.

\

\section{The category of multialgebras}  \label{sect-multi}

As mentioned in the Introduction,  the generalization to multialgebras of concepts from standard algebra such as homomorphism and subalgebras is not unique, and several choices are possible.

In this section the basic notions and results concerning the category of multialgebras, adopted here to be used along the paper, will be described (see also \cite{GC15} and~\cite{GC17}).

\begin{Notation} 
Let $A$ and $B$ be two sets. The set of all the functions $f:A \to B$ will be denoted by $B^A$. If $f:A \to B$ is a function, $X \subseteq A$ and $Y \subseteq B$ then $f[X]$ and $f^{-1}(Y)$ will stand for the sets $\{f(x) \ : \ x \in X\}$ and   
$\{x \in X \ : \ f(x) \in Y\}$, respectively.
If $\vec  a = (a_1 \ldots,a_n) \in A^n$ (for $n > 0$) then $f(\vec a)$ will stand for $(f(a_1),\ldots,f(a_n))$. If $\vec  b = (b_1 \ldots,b_n) \in B^n$ (for $n > 0$) then $f^{-1}(\vec b)$ will stand for $\{\vec a \in A^n \ : \ f(\vec a) = \vec b\}$. If $A$ is a nonempty set then $\wp(A)_+$ denotes the set of nonempty subsets of $A$.
\end{Notation}

\begin{definition} A {\em signature} is a denumerable family $\Sigma = \{\Sigma_n \ : \ n \geq 0\}$ of pairwise disjoint sets. Elements of $\Sigma_n$ are called {\em operator symbols of arity} $n$. Elements of $\Sigma_0$ are called {\em constants}.
\end{definition}

\begin{definition} Let $\Sigma$ be a signature. 
A {\em multialgebra} (or {\em hyperalgebra} or {\em non-deterministic algebra}) over $\Sigma$ is a pair $\mathcal{A}=(A,\sigma_\mathcal{A})$ such that $A$ is a nonempty set (the {\em support} of $\mathcal{A}$) and $\sigma_\mathcal{A}$  is a mapping assigning, to each $c \in \Sigma_n$, a function (called {\em multioperation} or {\em hyperoperation})  $c^\mathcal{A}:A^n \to \wp(A)_+$.  In particular, $\emptyset \neq c^\mathcal{A} \subseteq A$ if $c \in \Sigma_0$. 
\end{definition}

In the sequel, and when there is no risk of confusion, sometimes  we will refer to a multialgebra  $\mathcal{A}=(A,\sigma_\mathcal{A})$ by means of its support $A$. The support of  $\mathcal{A}$ will be frequently denoted by  $|\mathcal{A}|$.

\begin{definition} \label{subal}
Let $\mathcal{A}=(A,\sigma_\mA)$ and  $\mathcal{B}=(B,\sigma_\mB)$  be two multialgebras over $\Sigma$. Then $\mathcal{B}$ is said to be  a {\em submultialgebra} of $\mathcal{A}$, denoted by  $\mathcal{B} \subseteq \mathcal{A}$, if the following conditions hold:
\begin{enumerate}[(i)]
 \item $B\subseteq A$,
 \item if $c\in \Sigma_n$ and $\vec{a}\in B^n$, then $c^\mB(\vec{a}) \subseteq c^\mA(\vec{a})$; in particular, $c^\mB \subseteq c^\mA$ if $c \in \Sigma_0$. 
\end{enumerate}
\end{definition}

\begin{definition} \label{defHom}
Let $\mathcal{A}=(A,\sigma_\mA)$ and  $\mathcal{B}=(B,\sigma_\mB)$  be two multialgebras, and
let $f:A \to B$ be a function.\\[1mm]
(i)  $f$  is said to be a {\em homomorphism} from $\mathcal{A}$ to $\mathcal{B}$, denoted by $f: \mathcal{A} \to \mathcal{B}$, if $f[c^\mA(\vec{a})] \subseteq c^\mB(f(\vec{a}))$, for every  $c\in \Sigma_n$ and  $\vec{a}\in A^n$. In particular,  $f[c^\mA] \subseteq c^\mB$ for every $c \in \Sigma_0$.\\[1mm]
(ii) $f$ is said to be a {\em full homomorphism}  from $\mathcal{A}$ to $\mathcal{B}$, which is denoted by $f: \mathcal{A} \to_s \mathcal{B}$, if $f[c^\mA(\vec{a})] = c^\mB(f(\vec{a}))$ for every  $c\in \Sigma_n$ and  $\vec{a}\in A^n$. In particular,  $f[c^\mA] = c^\mB$ for every $c \in \Sigma_0$.
\end{definition}

\begin{remark}
If $\mathcal{B}$ and  $\mathcal{A}$  are two multialgebras over $\Sigma$ such that $|\mathcal{B}| \subseteq |\mathcal{A}|$ then: $\mathcal{B} \subseteq \mathcal{A}$ iff the inclusion map $i:|\mathcal{B}| \to |\mathcal{A}|$ is a homomorphism from $\mathcal{B}$ to $\mathcal{A}$.
\end{remark}

Observe that, if $f: |\mathcal{A}| \to |\mathcal{B}|$ and $g:|\mathcal{B}| \to |\mathcal{C}|$ are homomorphisms of multialgebras then $g \circ f:|\mathcal{A}| \to |\mathcal{C}|$ is also a homomorphism of multialgebras. On the other hand, the identity mapping $i_A:A \to A$ is a homomorphism from $\mathcal{A}$ to $\mathcal{A}$, for every multialgebra $\mathcal{A}=(A,\sigma_\mA)$. This means that there is a category of multialgebras over $\Sigma$ and their morphisms, that will be called  $\malg(\Sigma)$.

The following results will be useful in the sequel:

\begin{proposition} \label{isos}
Let $\mathcal{A}=(A,\sigma_\mA)$ and  $\mathcal{B}=(B,\sigma_\mB)$  be two multialgebras over $\Sigma$, and
let $f:A\to B$ be a function. Then, $f$ is an isomorphism  $f:\mathcal{A}\to \mathcal{B}$  in the category $\malg(\Sigma)$ iff $f$ is a full homomorphism  $f:\mathcal{A}\to_s \mathcal{B}$ which is a bijective function. 
\end{proposition}
\begin{proof}
It is an immediate consequence of the definitions.
\end{proof}

\begin{proposition} \label{monos}
Let $\mathcal{A}=(A,\sigma_\mA)$ and  $\mathcal{B}=(B,\sigma_\mB)$  be two multialgebras over $\Sigma$, and
let  $f:\mathcal{A}\to \mathcal{B}$ be a homomorphism. If $f:A\to B$ is an injective function then  $f$ is a monomorphism in the category $\malg(\Sigma)$. 
\end{proposition}
\begin{proof}
It is also an immediate consequence of the definitions.
\end{proof}

\begin{proposition} \label{epis}
Let $\mathcal{A}=(A,\sigma_\mA)$ and  $\mathcal{B}=(B,\sigma_\mB)$  be two multialgebras over $\Sigma$, and
let $f:A\to B$ be a function. Then, $f$ is an epimorphism  $f:\mathcal{A}\to \mathcal{B}$  in the category $\malg(\Sigma)$ iff  $f$ is a homomorphism in $\malg(\Sigma)$ such that  $f$ is a surjective function. 
\end{proposition}
\begin{proof}
If $f$ is a surjective homomorphism then it is clear that it is an epimorphism in $\malg(\Sigma)$. Conversely, suppose that $f:\mathcal{A}\to \mathcal{B}$ is an epimorphism in 
$\malg(\Sigma)$ and let $\mathcal{A}'$ be a multialgebra over $\Sigma$ with domain $\{0,1\}$ such that $c^{\mathcal{A}'}(\vec a)=\{0,1\}$ for every $c \in \Sigma_n$ and $\vec{a} \in \{0,1\}^n$; in particular, $c^{\mathcal{A}'}=\{0,1\}$ for every $c \in \Sigma_0$. Consider $g:B \to \{0,1\}$ such that $g(x)=1$ if there exists $y \in A$ such that $x=f(y)$, and $g(x)=0$ otherwise. Clearly, $g$ is a homomorphism $g:\mathcal{B} \to \mathcal{A}'$ in $\malg(\Sigma)$. Finally, let  $h:B \to \{0,1\}$ such that $h(x)=1$ for every $x \in B$. It is also clear that $h$ is a homomorphism $g:\mathcal{B} \to \mathcal{A}'$ in $\malg(\Sigma)$. Since $g\circ f = h\circ f$ and $f$ is epimorphism in $\malg(\Sigma)$ then $g=h$. This means that $f$ is a surjective function.
\end{proof}

\begin{proposition} \label{prods}
The category $\malg(\Sigma)$ has arbitrary products.
\end{proposition}
\begin{proof} Let $\{\mathcal{A}_i \ : \ i \in I\}$ be a family of multialgebras over $\Sigma$. If $I=\emptyset$ then the result is obvious: the multialgebra $\mathbf{1} =(\{\ast\},\sigma_{\bf 1})$ such that $c^{\bf 1}(\ast, \ldots, \ast)=\{\ast\}$ for every $c \in \Sigma_n$ (with $n> 0$) and $c^{\bf 1}=\{\ast\}$ for every $c \in \Sigma_0$ is the terminal object in  $\malg(\Sigma)$. Now, assume that $I \neq \emptyset$, and let $A=\prod_{i \in I} A_i$ be the standard construction of the cartesian product of the family of sets $\{A_i \ : \ i \in I\}$ with canonical projections $\pi_i:A \to A_i$ for every $i \in I$. That is, $A =\big\{a \in \big(\bigcup_{i\in I} A_i\big)^I \ : \  a(i)\in A_i \ \mbox{ for every $i \in I$}\big\}$ and, for every $i \in I$ and every $a \in A$, $\pi_i(a) = a(i)$. Consider the multialgebra  $\mathcal{A}=(A,\sigma_\mA)$ over $\Sigma$ such that, for every $c \in \Sigma_n$ and every $\vec a \in A^n$, $c^\mA(\vec a) = \prod_{i \in I} c^{\mA_i}(\pi_i(\vec a))$. In particular, $c^\mA = \prod_{i \in I} c^{\mA_i}$ for every $c \in \Sigma_0$.
It is easy to see that each $\pi_i$ is a (full) homomorphism from  $\mathcal{A}$ to $\mathcal{A}_i$ such that $\langle \mathcal{A}, \{\pi_i \ : \ i \in I\}\rangle$ is the product in  $\malg(\Sigma)$  of the family $\{\mathcal{A}_i \ : \ i \in I\}$.
\end{proof}

\begin{definition} \label{imdir}
Let $\mathcal{A}=(A,\sigma_\mA)$ and  $\mathcal{B}=(B,\sigma_\mB)$  be two multialgebras over $\Sigma$, and
let  $f: \mathcal{A} \to \mathcal{B}$ be a homomorphism in $\malg(\Sigma)$. The {\em direct image of $f$} is the submultialgebra $f(\mathcal{A})=(f[A],\sigma_{f(\mathcal{A})})$ of  $\mathcal{B}$ such that, for every $c \in \Sigma_n$ and $\vec b \in f[A]$, $c^{f(\mathcal{A})}(\vec b) = \bigcup\big\{f[c^\mathcal{A}(\vec a)] \ : \ \vec a \in f^{-1}(\vec b) \big\}$. In particular, $c^{f(\mathcal{A})} = f[c^\mathcal{A}]$ for every $c \in \Sigma_0$. 
\end{definition}

Observe that, if $\vec b \in f[A]$ and $\vec a \in f^{-1}(\vec b)$ then $f[c^\mathcal{A}(\vec a)] \subseteq c^\mathcal{B}(f(\vec a)) =   c^\mathcal{B}(\vec b)$ whence $c^{f(\mathcal{A})}(\vec b) \subseteq c^\mathcal{B}(\vec b)$. This means that $f(\mathcal{A})$ is, indeed, a submultialgebra of $\mathcal{B}$. Moreover, the following useful result holds in $\malg(\Sigma)$:

\begin{proposition} [Epi-mono factorization] \label{epi-mono}
Consider  two multialgebras $\mathcal{A}=(A,\sigma_\mA)$ and  $\mathcal{B}=(B,\sigma_\mB)$  over $\Sigma$, and
let  $f:\mathcal{A}\to \mathcal{B}$ be a homomorphism in $\malg(\Sigma)$. Let $\bar f:A \to f[A]$ be the mapping given by $\bar f(x)=f(x)$ for every $x \in A$, and let $g:f[A] \to B$ be the inclusion map. Then $\bar f$ and $g$ are homomorphisms $\bar f:\mathcal{A}\to f(\mathcal{A})$ and  $g:f(\mathcal{A}) \to \mathcal{B}$ such that $\bar f$ is an epimorphism in $\malg(\Sigma)$, $g$ is a monomorphism in $\malg(\Sigma)$, and $f=g \circ \bar f$. 
$$\xymatrix{\mathcal{A} \ar[rr]^{\hspace*{-2mm}f} \ar[rrd]_{\bar f} && \mathcal{B}\\
&&  f(\mathcal{A})  \ar@{^{(}->}[u]_g}$$
Moreover, if $f$ is injective (as a function)  then $\bar f$ is an isomorphism in $\malg(\Sigma)$.
\end{proposition}
\begin{proof}
It is immediate from the previous results.
\end{proof}
\

It is important to observe  that our epi-mono factorization could not be unique (up to isomorphism).

\

\begin{definition} \label{mcongru}
Let   $\mathcal{A}=(A,\sigma_\mA)$ be a multialgebra, and let $\Theta \subseteq A\times A$. Then $\Theta$ is said to be a {\em multicongruence} over $\mathcal{A}$ if the following properties hold:
\begin{enumerate}[(i)]
 \item $\Theta$ is an equivalence relation;
 \item for every $n>0$, $c\in \Sigma_n$  and $\vec a, \vec b\in A^n$: if $(a_i,b_i) \in \Theta$ for every $1 \leq i \leq n$ then, for every  
 $a\in c^\mathcal{A}(\vec a)$ there is $b \in c^\mathcal{A}(\vec b)$ such that $(a,b)\in \Theta$;
 \item for every  $c\in \Sigma_0$ and every $a,b \in A$: if $a,b \in c^\mathcal{A}$ then $(a,b)\in \Theta$.
\end{enumerate}
\end{definition}

\begin{definition} \label{quomultialg}
Let   $\mathcal{A}=(A,\sigma_\mA)$ be a multialgebra, and let $\Theta$ be a multicongruence over  $\mathcal{A}$. The {\em quotient multialgebra} (or {\em factor multialgebra}) of $\mathcal{A}$ modulo $\Theta$ is the multialgebra  $\mathcal{A}/_\Theta=(A/_\Theta,\sigma_{\mathcal{A}/_\Theta})$ such that, for every $c \in \Sigma_n$ and every $(a_1/_\Theta,\ldots,a_n/_\Theta) \in (A/_\Theta)^n$, $c^{\mathcal{A}/_\Theta}(a_1/_\Theta,\ldots,a_n/_\Theta) = \big\{ a/_\Theta \ : \  a \in c^{\mathcal{A}}(\vec a)\big\}$. In particular, $c^{\mathcal{A}/_\Theta} = \big\{ a/_\Theta \ : \  a \in c^{\mathcal{A}}\big\}$  for every $c \in \Sigma_0$. The canonical map $p:A \to A/_\Theta$ is given by  $p(a)=a/_\Theta$ for every $a \in A$.
\end{definition}

\begin{proposition} \label{canproj-quo-epi}
Let   $\mathcal{A}=(A,\sigma_\mA)$ be a multialgebra, and let $\Theta$ be a multicongruence over  $\mathcal{A}$. Then  $\mathcal{A}/_\Theta$  is a multialgebra, and the canonical map $p:A \to A/_\Theta$  determines a (full) homomorphism of multialgebras $p:\mA \to \mA/_\Theta$ such that $p(\mathcal{A})=\mathcal{A}/_\Theta$.
\end{proposition}

\section{From \cplp\ to the logic \mbc} \label{sect-mbC}

The class of paraconsistent logics known as {\em Logics of Formal Inconsistency} (\lfis, for short) was introduced by W. Carnielli and J. Marcos in~\cite{CM}. In its simplest form, they have a non-explosive negation $\neg$, as well as  a (primitive or derived) {\em consistency connective} $\circ$ which allows to recover the explosion law in a controlled way. 

\begin{definition} \label{defLFI}
Let  ${\bf L}=\langle  \Theta,\vdash \rangle$ be a Tarskian, finitary and structural logic defined over a propositional signature $\Theta$, which contains a
negation $\neg$, and let  $\circ$ be a (primitive or defined) unary connective. Then,
${\bf L}$ is said to be  a {\em Logic of Formal Inconsistency} with respect to $\neg$ and $\circ$ if the following holds:
\begin{itemize}
       \item [(i)] $\varphi,\neg\varphi\nvdash\psi$ for some $\varphi$ and $\psi$;
       \item [(ii)] there are two formulas $\alpha$ and $\beta$ such that
     \begin{itemize}
       \item [(ii.a)] $\circ\alpha,\alpha \nvdash \beta$;
       \item [(ii.b)] $\circ\alpha, \neg \alpha \nvdash \beta$; 
\end{itemize}
       \item [(iii)]  $\circ\varphi,\varphi,\neg\varphi\vdash\psi$ for every $\varphi$ and $\psi$. 
\end{itemize}
\end{definition}

\

Condition (ii) of the definition of \lfis\ is required in order to satisfy condition~(iii) in  a non-trivial way. The hierarchy of \lfis\ studied in~\cite{CCM} and~\cite{CC16} starts from a logic called \mbc, which extends  positive classical logic \cplp\ by adding a negation $\neg$ and an unary {\em consistency} operator $\circ$ satisfying minimal requirements in order to define an \lfi. 

From now on, the following three signatures will be mainly considered:
\begin{enumerate}
\item[] $\Sigma_+=\{\land, \lor, \to\}$;
\item[] $\Sigma_{\rm BA}=\{\land, \lor, \to,0,1\}$; and
\item[] $\Sigma=\{\land, \lor, \to, \neg, \cons\}$.
\end{enumerate}

If $\Theta$ is a propositional signature, then $For(\Theta)$ will denote the (absolutely free) algebra of formulas over $\Theta$ generated by a given denumerable set $\mathcal{V}=\{p_n \ : \  n \in \mathbb{N}\}$ of propositional variables.

\begin{definition} [Classical Positive Logic] The {\em classical  positive logic} \cplp\ is defined over the language $For(\Sigma_+)$ by the following Hilbert  calculus:\\[2mm]
  {\bf Axiom schemas:}
  \begin{gather}
    \alpha \imp \big(\beta \imp \alpha\big)             \tag{\kax} \\
   \Big(\alpha\imp\big(\beta\imp\gamma\big)\Big) \imp
            \Big(\big(\alpha\imp\beta\big)\imp\big(\alpha\imp\gamma\big)\Big)
              \tag{\axTrans}\\
    \alpha \imp \Big(\beta \imp \big(\alpha \land \beta\big)
                                             \Big)  \tag{\axed}\\
    \big(\alpha \land \beta\big) \imp \alpha         \tag{\axeea}\\
    \big(\alpha \land \beta\big) \imp \beta          \tag{\axeeb}\\
    \alpha \imp \big(\alpha \lor \beta\big)          \tag{\axouda}\\
    \beta \imp \big(\alpha \lor \beta\big)           \tag{\axoudb}\\
    \Big(\alpha \imp \gamma\Big) \imp \Big(
        (\beta \imp \gamma) \imp
           \big(
          (\alpha \lor \beta) \imp \gamma
           \big)\Big)                               \tag{\axoue} \\
    \big(\alpha \imp \beta\big) \lor \alpha          \tag{\axouimp}
  \end{gather}

{\bf Inference rule:}
\[\frac{\alpha \ \ \ \ \alpha\imp
      \beta}{\beta}  \tag{\MP}\]
\end{definition}
 
\

\begin{definition} The logic \mbc, defined over signature $\Sigma$, is obtained from  \cplp\ by adding the following axiom schemas:
  \begin{gather}
    \alpha \lor \lnot \alpha                        \tag{\axtnd}\\
    \cons \alpha \imp \Big(\alpha \imp \big(\lnot \alpha \imp \beta\big)\Big)
                                                     \tag{\axexp}
  \end{gather}
\end{definition}

\

For convenience, the expansion of   \cplp\ over signature $\Sigma$ will be considered from now on, besides \cplp\ itself. This logic, denoted by \ecpl, is nothing more than  \cplp\ defined over $\Sigma$  by adding $\neg$ and $\circ$ as additional unary connectives without any axioms or rules for them.

\section{Swap structures for \ecpl} \label{sect-sw}

In  \cite{CC16} was introduced the notion of swap structures for \mbc, as well as for some axiomatic extensions of it.
In this section, these structures will be reintroduced in a slightly more general form, in order to define a hierarchy of classes of multialgebras associated to the corresponding hierarchy of logics. This is in line with the traditional approach of algebraic logic, in which hierarchies of  classes of algebraic models are  associated to hierachies of logics.
From now on, $\Sigma$ will denote the signature for \mbc.

Since \mbc\ is an axiomatic extension of \ecpl, it is natural to begin with swap structures for the latter logic.
Recall the following:

\begin{definition} An {\em implicative lattice} is an algebra $\mathcal{A}=\langle A,\land,\lor,\imp\rangle$ where $\langle A,\land,\lor\rangle$ is a lattice  such that $\bigvee \{ c \in A \ : \  a \wedge c \leq b\}$ exists for every $a,b \in A$,\footnote{Here, $\leq$ denotes the partial order associated with the lattice, namely: $a \leq b$ iff $a=a\wedge b$ iff $b=a \vee b$, and $\bigvee X$ denotes the supremum of the set $X \subseteq A$ w.r.t. $\leq$, provided that it exists.} and $\imp$ is  the induced  implication given by $a \imp b =\bigvee \{ c \in A \ : \  a \wedge c \leq b\}$ for every $a,b \in A$ (note that $1\defin a \to a$ is the  top element of $A$,  for any $a \in A$). 
If, additionally,  $a \vee (a \imp b)=1$ for every $a,b$ then $\mathcal{A}$ is said to be a {\em \algcpl}.\footnote{The name was taken from H. Curry, see~\cite{C77}.}
\end{definition}

The following results are well-known:

\begin{proposition} \label{BA-HA} Let $\mathcal{A}$ be an implicative lattice. Then:\\
(1) If $\mathcal{A}$ has a bottom element $0$, then it is  a Heyting algebra.\\
(2) If $\mathcal{A}$ is a  \algcpl\ and it  has a bottom element $0$, then it is  a Boolean algebra.
\end{proposition} 

The algebraic semantics for \cplp\ is given by  classical implicative lattices. In formal terms:

\begin{theorem} \label{adeq-CPLP}
Let $\Gamma \cup \{\alpha\}$ be a set of formulas over the signature $\Sigma_+$. Then:
 $\Gamma \vdash_{\cplp} \alpha$ iff, for every \algcpl\ $\mathcal{A}$ and for  every homomorphism $h:For(\Sigma_+) \to \mathcal{A}$, if $h(\gamma)=1$ for every $\gamma \in \Gamma$ then $h(\alpha)=1$.
\end{theorem}

Now, a semantics of multialgebras of triples over a given  Boolean algebra $\mathcal{A}$, which are called {\em swap structures}, will be introduced for \ecpl. The idea is that a triple $(z_1,z_2,z_3)$  in such structures represents a (complex) truth-value in which $z_1$ interprets a given truth-value for a formula $\alpha$, while $z_2$ and $z_3$ represent a possible truth-value for $\neg \alpha$ and $\circ \alpha$, respectively.
The reason to take a Boolean algebra instead of a \algcpl\ is  the following: given an \lfi\ extending \ecpl, in order to prove completeness w.r.t. swap structures a \algcpl\ is naturally defined by means of a Lindenbaum-Tarski process. Since any \lfi\ can define a bottom formula, the obtained \algcpl\  becomes a Boolean algebra, by Proposition~\ref{BA-HA}(2).
In the case of \ecpl, a technical result (see propositions~\ref{propA*} and~\ref{propA*2} below) will allow to extend each  \algcpl\ to a Boolean algebra.

Let $\mathcal{A}=\langle A, \wedge, \vee, \to,0,1 \rangle$ be a Boolean algebra and let $\pi_{(j)}:A^3 \to A$ be the canonical projections, for $1 \leq j \leq 3$.
Observe that,  if $z \in A^3$ and $z_j=\pi_{(j)}(z)$ for $1 \leq j \leq 3$ then $z=(z_1,z_2,z_3)$.

\begin{definition} \label{univ-sw}
Let $\mathcal{A}$ be a Boolean algebra with domain $A$. The universe of swap structures for \ecpl\ over $\mathcal{A}$  is the set $\textsc{B}_\mathcal{A}^{\ecpl}= A^3$.
\end{definition}

\begin{definition} \label{Swap-str}
Let $\mathcal{A}=\langle A, \wedge, \vee, \to,0,1 \rangle$ be a Boolean algebra, and let $ B\subseteq \textsc{B}_\mathcal{A}^{\ecpl}$.
A {\em swap structure for \ecpl\ over $\mathcal{A}$} is any multialgebra $\mathcal{B}=\langle B, \wedge_\mathcal{B}, \vee_\mathcal{B},$\\ 
$\to_\mathcal{B},\neg_\mathcal{B}, \circ_\mathcal{B} \rangle$ over $\Sigma$ such that $0 \in \pi_1[B]$ and the multioperations satisfy the following, for every $z$ and $w$ in $B$:

\begin{itemize}
 \item[(i)] $\emptyset \neq z\#_\mathcal{B}w \subseteq \{u\in B \ : \ u_1=z_1\#w_1\}$, for each $\#\in \{\wedge,\vee, \to\}$;
\item[(ii)] $\emptyset \neq \neg_\mathcal{B} (z) \subseteq \{u\in B \ : \ u_1=z_2\}$;
 \item[(iii)] $\emptyset \neq \circ_\mathcal{B} (z) \subseteq \{u\in B \ : \ u_1=z_3\}$.
\end{itemize}
\end{definition}

When there is no risk of confusion, the subscript `$\mathcal{B}$' will be omitted when referring to the multioperations of  $\mathcal{B}$.

\begin{definition}
Let  $\Ksw$ be the class of swap structures for \ecpl.
The full subcategory in  $\malg(\Sigma)$ of swap structures for \ecpl\ will be denoted by \sw.
\end{definition}

From the previous definition,  the class of objects of \sw\  is \Ksw, and the morphisms between two given swap structures are just the homomorphisms between them as multialgebras.

\begin{definition} \label{BmaxCPL}
Let  $\mathcal{A}$ be a  Boolean algebra. The {\em full swap structure for \ecpl\ over $\mathcal{A}$}, denoted by $\mathcal{B}_\mathcal{A}^{\ecpl}$, is the unique swap structure for \ecpl\ over $\mathcal{A}$  with domain $\textsc{B}_\mathcal{A}^{\ecpl}= A^3$ such that, for every $z$ and $w$ in $A^3$:

\begin{itemize}
 \item[(i)] $z\#w = \{u\in A^3 \ : \ u_1=z_1\#w_1\}$, for each $\#\in \{\wedge,\vee, \to\}$;
\item[(ii)] $\neg (z) = \{u\in A^3 \ : \ u_1=z_2\}$;
 \item[(iii)] $\circ (z) = \{u\in A^3 \ : \ u_1=z_3\}$.\\[-2mm]
\end{itemize}
\end{definition}

\begin{remark} \label{obs pi(B)}
The term ``full'' is adopted in Definition~\ref{BmaxCPL} in analogy with the terminology used by S. Odintsov in~\cite{odin:03} with respect to twist structures. This is justified by the fact that swap structures can be considered as non-deterministic twist structures (or, from the opposite perspective, twist structures are particular cases of swap structures), as it will be argued in Section \ref{sectwist-swap}.

Observe that, if $\mathcal{B}$ is a swap structure for \ecpl\ over $\mathcal{A}$, then $\mathcal{B}$ is a submultialgebra of $\mathcal{B}_\mathcal{A}^{\ecpl}$ in the sense of Definition~\ref{subal}. Thus, $\mathcal{B}_\mathcal{A}^{\ecpl}$ is the greatest swap structure for \ecpl\ over $\mathcal{A}$. 
\end{remark}

\begin{proposition} \label{pi1-BA}
Let $\mathcal{B}$ be a swap structure for \ecpl\ over $\mathcal{A}$ and let $\mathbb{A}(\mathcal{B}) \defin$\\
$\pi_1[|\mathcal{B}|]$. Then, $\mathbb{A}(\mathcal{B})$ is a  Boolean subalgebra of $\mathcal{A}$.   Moreover, $\mathbb{A}\big(\mathcal{B}_\mathcal{A}^{\ecpl}\big) = \mathcal{A}$.
\end{proposition}
\begin{proof} Let $\mathcal{B}$ be a swap structure for \ecpl\ over $\mathcal{A}$.
For each $a \in \mathbb{A}(\mathcal{B})$ choose an element $z(a)$ in $|\mathcal{B}|$ such that $\pi_1(z(a))=a$.
Observe that $0 \in \mathbb{A}(\mathcal{B})$, by Definition~\ref{Swap-str}. Since $|\mathcal{B}|$ is closed under the multioperations of $\mathcal{B}$ then $z(0) \to z(0) \subseteq |\mathcal{B}|$ and so $\{1\}=\pi_1[z(0) \to z(0)] \subseteq \mathbb{A}(\mathcal{B})$. That is, $1 \in \mathbb{A}(\mathcal{B})$.

For  each $\#\in \{\wedge,\vee, \to\}$ observe that  $\pi_1[z(a)\#z(b)] = \{a \# b\}$ for every $a,b \in \mathbb{A}(\mathcal{B})$, by Definition~\ref{Swap-str}. This means that $a \# b \in \mathbb{A}(\mathcal{B})$ for every $a,b \in \mathbb{A}(\mathcal{B})$ and for  each $\#\in \{\wedge,\vee, \to\}$. Therefore $\mathbb{A}(\mathcal{B})$ is a  Boolean  subalgebra of $\mathcal{A}$.
\end{proof}

\

Elements of a swap structure for \ecpl\ are called {\em snapshots for \ecpl}.   
Since no axioms or rules are given in \ecpl\ for the unary connectives  $\neg$ and $\circ$, the multioperations associated to them in a swap structure just put in evidence (or `swap') on the first coordinate  the corresponding value, leaving free the values of the other coordinates. This produces two (nonempty) sets of snapshots, defining so  multioperations for the conectives $\neg$ and $\circ$. As we shall see in the next sections, when axioms are considered for these unary connectives, the multioperations (and the domain of the swap structures themselves) must be restricted accordingly, obtaining so different classes of multialgebras.

\

\section{Swap structures semantics for \ecpl}

Recall the semantics associated to Nmatrices introduced by A. Avron and I. Lev:

\begin{definition} [\cite{avr:lev:01}] \label{valNmat} 
Let $\mathcal{M}=(\mathcal{B},D)$ be an Nmatrix over a signature $\Theta$. A {\em valuation} over $\mathcal{M}$ is a function $v: For(\Theta)\to |\mathcal{B}|$ such that, for every $c \in \Theta_n$ and every $\varphi_1,\ldots,\varphi_n \in For(\Theta)$:
$$v(c(\varphi_1,\ldots,\varphi_n)) \in c^\mathcal{B}(v(\varphi_1),\ldots,v(\varphi_n)).$$
In particular, $v(c) \in c^\mathcal{B}$, for every $c \in \Theta_0$.
\end{definition}

\begin{definition}  \label{semNmat} 
Let $\mathcal{M}=(\mathcal{B},D)$ be an Nmatrix over a signature $\Theta$, and let $\Gamma \cup \{\varphi\} \subseteq For(\Theta)$. We say that $\varphi$ is a consequence of $\Gamma$  in the Nmatrix $\mathcal{M}$, denoted by $\Gamma\models_{\mathcal{M}} \varphi$, if the following holds: for every valuation  $v$ over $\mathcal{M}$, if $v[\Gamma] \subseteq D$ then $v(\varphi) \in D$. In particular, $\varphi$ is valid in $\mathcal{M}$, denoted by $\models_{\mathcal{M}} \varphi$, if $v(\varphi) \in D$ for every valuation  $v$ over $\mathcal{M}$.
\end{definition}

The generalization of Nmatrix semantics to classes of Nmatrices is immediate:

\begin{definition} \label{conseqclassNmat}
Let $\mathbb{M}$ be a nonempty class of Nmatrices over a signature $\Theta$, and let  $\Gamma \cup \{\varphi\} \subseteq For(\Theta)$  be a set of formulas over $\Theta$. We say that $\varphi$ is a consequence of $\Gamma$  in the class $\mathbb{M}$ of Nmatrices, denoted by $\Gamma\models_{\mathbb{M}} \varphi$, if  $\Gamma\models_{\mathcal{M}} \varphi$ for every $\mathcal{M} \in \mathbb{M}$.  In particular, $\varphi$ is valid in $\mathbb{M}$, denoted by $\models_{\mathbb{M}} \varphi$,  if it is valid in every  $\mathcal{M} \in \mathbb{M}$. 
\end{definition}

\begin{remark}
Given  a signature $\Theta$, the (absolutely free) algebra of formulas  $For(\Theta)$ over $\Theta$ generated by the set $\mathcal{V}$ of propositional variables can be considered as a multialgebra $\mathcal{F}or(\Theta)$ over $\Theta$ in which the multioperators (the conectives of  $\Theta$ themselves) are single-valued. That is, $c^{\mathcal{F}or(\Theta)}(\alpha_1,\ldots,\alpha_n) \defin \{c(\alpha_1,\ldots,\alpha_n)\}$ for every $n$-ary connective $c \in \Theta$ and every $\alpha_1,\ldots,\alpha_n \in For(\Theta)$. Being so, it is interesting to notice that a valuation $v:For(\Theta) \to |\mathcal{B}|$ over an Nmatrix $\mathcal{M}=(\mathcal{B},D)$ in the sense of Definition~\ref{defHom}(i) is an homomorphism $v:\mathcal{F}or(\Theta) \to \mathcal{B}$ in the category $\malg(\Theta)$ of multialgebras. This means that the semantics of  Nmatrices constitutes a genuine generalization of the standard matrix semantics, provided that the category of multiagebras into consideration is precisely $\malg(\Theta)$.
\end{remark}

\

Recall that  $\Ksw$ denotes the class of swap structures for \ecpl. As it was done in~\cite[Chapter~6]{CC16} with several \lfis, it is easy to see that each $\mathcal{B} \in \Ksw$ induces naturally a non-deterministic matrix  such that the class of such Nmatrices semantically characterizes \ecpl.   More precisely:

\begin{definition} \label{Nmatrix}
For each $\mathcal{B} \in \Ksw$ let $D_\mathcal{B} = \{z \in |\mathcal{B}| \ :  \  z_1=1\}$. The {\em Nmatrix associated to} $\mathcal{B}$ is $\mathcal{M}(\mathcal{B})=(\mathcal{B}, D_\mathcal{B})$.  Let $$Mat(\Ksw) = \big\{\mathcal{M}(\mathcal{B}) \ : \ \mathcal{B} \in  \Ksw\big\}.$$
\end{definition}

In this particular case, Definition~\ref{valNmat} assumes the following form:

\begin{definition} \label{valNmat1} Let $\mathcal{B} \in \Ksw$ and $\mathcal{M}(\mathcal{B})$ as above. A {\em valuation} over $\mathcal{M}(\mathcal{B})$ is a function $v: For(\Sigma)\to |\mathcal{B}|$ such that, for every $\varphi_1,\varphi_2 \in For(\Sigma)$:
\begin{enumerate}[(i)]
 \item  $v(\varphi_1 \#\varphi_2) \in v(\varphi_1) \# v(\varphi_2)$, for every $\#\in \{\wedge,\vee, \to\}$;
 \item $v(\neg\varphi_1) \in \neg v(\varphi_1)$;
 \item $v(\circ\varphi_1) \in \circ v(\varphi_1)$.
\end{enumerate}
\end{definition}

In order to prove the adequacy of \ecpl\ w.r.t. swap structures (that is, w.r.t. the class  $Mat(\Ksw)$ of Nmatrices, by using Definition~\ref{conseqclassNmat}), some previous technical results must be obtained. Given a \algcpl\ $\mathcal{A}$, it is always possible to formally ``duplicate'' $\mathcal{A}$ by considering $A^* \defin A \times \{0,1\}$ such that, for any $a \in A$, the pairs $(a,1)$ and $(a,0)$ can be considered in $A^*$ as representing uniquely $a$ and its Bolean complement ${\sim}a$, respectively. In formal terms:

\begin{definition} \label{defoperA*}
Let $\mathcal{A}=\langle A, \wedge, \vee, \to \rangle$ be a \algcpl, and let $A^* \defin A \times \{0,1\}$. Consider the operations $\land$, $\lor$ and $\imp$ defined over $A^*$ as  follows, for every $a,b \in A$:
\begin{itemize}
\item[] $(a,1) \# (b,1) = (a \#b, 1)$,  for $\# \in \{\land, \lor, \imp\}$;
\item[] $(a,1) \land (b,0) = (b,0) \land (a,1) = (a \imp b,0)$;
 \item[] $(a,0) \land (b,0) = (a \lor b,0)$;
\item[] $(a,1) \lor (b,0) = (b,0) \lor (a,1) = (b \imp a,1)$;
 \item[] $(a,0) \lor (b,0) = (a \land b,0)$;
\item[] $(a,1) \imp (b,0) = (a \land b,0)$;
\item[] $(a,0) \imp (b,1) = (a \lor b,1)$;
 \item[] $(a,0) \imp (b,0) = (b \imp a,1)$.
\end{itemize}
\end{definition}

\begin{proposition} \label{propA*}
The structure $\mathcal{A}^* =\langle A^*, \wedge, \vee, \to, 0^*,1^* \rangle$, where the binary operators $\{\land, \lor, \imp\}$ are defined as in Definition~\ref{defoperA*}, is a Boolean algebra such that $0^* \defin(1,0)$ and $1^* \defin(1,1)$.
\end{proposition}
\begin{proof}
By considering $(a,1)$ and $(a,0)$  as representing  in $A^*$ the elements $a$ of $A$ and its Bolean complement ${\sim}a$, respectively, the proof is straightforward.
\end{proof}

\begin{proposition} \label{propA*2}
Given a \algcpl\ $\mathcal{A}$, let $\mathcal{A}^*$ as in Proposition~\ref{propA*}.\\
(1) Let $i^*:A \to A^*$ be the mapping given by $i^*(a)=(a,1)$, for every $a \in A$. Then $i^*$ is a monomorphism of classical implicative lattices.\\
(2) The pair $(\mathcal{A}^*,i^*)$ has the following universal property: if $\mathcal{A}'$ is a Boolean algebra and $h:\mathcal{A} \to \mathcal{A}'$ is a homomorphism of classical implicative lattices then there exists a unique homomorphism of Boolean algebras $h^*:\mathcal{A}^* \to \mathcal{A}'$ such that $h=h^* \circ i^*$. That is, the diagram below commutes.
$$\xymatrix{\mathcal{A} \ar@{^{(}->}[rr]^{\hspace*{-2mm}{i^*}} \ar[rrd]_{h} && \mathcal{A}^*\ar@{.>}[d]^{{h^*}}\\
&&  \mathcal{A}'}$$
\end{proposition}
\begin{proof}
Let $h^*(a,1)=h(a)$ and $h^*(a,0)={\sim}h(a)$ for every $a \in A$, where $\sim$ denotes the Boolean complement in $\mathcal{A}'$. The details of the proof are left to the reader.
\end{proof}

Consider now the consequence relation $\models_{Mat(\Ksw)}$ as in Definition~\ref{conseqclassNmat}, generated by the class $Mat(\Ksw)$ of Nmatrices associated to swap structures for \ecpl. Thus:

\begin{theorem} [Adequacy of \ecpl\ w.r.t. swap structures] \label{adeq-CPLP1} \ \\
Let $\Gamma \cup \{\varphi\} \subseteq For(\Sigma)$  be a set of formulas of \ecpl. Then:
$\Gamma \vdash_{\ecpl} \varphi$ \ iff \  $\Gamma\models_{Mat(\Ksw)} \varphi$.
\end{theorem}
\begin{proof}
`Only if' part (Soundness):
Observe that, if $v$ is a valuation over a swap structure $\mathcal{B}$ for \ecpl\ then $h=\pi_1\circ v:For(\Sigma) \to \mathcal{A}$ is a $\Sigma_+$-homomorphism such that $h(\gamma)=1$ iff $v(\gamma) \in D_\mathcal{B}$, by the very definitions.
Thus, suppose that $\Gamma \vdash_{\ecpl} \varphi$, and let  $v$ is a valuation over  $\mathcal{B} \in \Ksw$ such that $v[\Gamma] \subseteq D_\mathcal{B}$. As observed above, $h=\pi_1\circ v$ is a $\Sigma_+$-homomorphism such that $h[\Gamma] \subseteq \{1\}$ and so, by Theorem~\ref{adeq-CPLP}, $h(\varphi)=1$. Hence $v(\varphi) \in D_\mathcal{B}$, showing that  $\Gamma\models_{Mat(\Ksw)} \varphi$.\\ [2mm]
`If' part (Completeness): Suppose that $\Gamma \nvdash_{\ecpl} \varphi$. Define in $For(\Sigma)$ the following relation: $\alpha \equiv_\Gamma \beta$ iff $\Gamma \vdash_{\ecpl} \alpha \imp \beta$ and $\Gamma \vdash_{\ecpl} \beta \imp \alpha$. It is clearly an equivalence relation. Let $A_\Gamma \defin For(\Sigma)/_{\equiv_\Gamma}$ be the quotient set, and define over $A_\Gamma$ the following operations: $[\alpha]_{\Gamma} \,\#\, [\beta]_{\Gamma} \defin [\alpha \# \beta]_{\Gamma}$, for $\# \in \{\land,\lor,\imp\}$  (here, $[\alpha]_\Gamma$ denotes the equivalence class of $\alpha$ w.r.t. $\equiv_\Gamma$). These operations are clearly well-defined, and so they induce a structure of \algcpl\ over the set   $A_\Gamma$. Let $\mathcal{A}_\Gamma$ be the obtained \algcpl, and let
 $(\mathcal{A}_\Gamma)^*$ be the Boolean algebra induced by $\mathcal{A}_\Gamma$ as in Definition~\ref{defoperA*}. Let  $\mathcal{B}_{(\mathcal{A}_\Gamma)^*}^{\ecpl}$ be the corresponding swap structure in \Ksw\ as in Definition~\ref{BmaxCPL}, and let $\mathcal{M}_\Gamma^{\ecpl} \defin\mathcal{M}(\mathcal{B}_{(\mathcal{A}_\Gamma)^*}^{\ecpl})$. Consider now a mapping $v_\Gamma^*:For(\Sigma) \to   (A_\Gamma^*)^3$ given by $v_\Gamma^*(\alpha) =(([\alpha]_{\Gamma},1),([\neg \alpha]_{\Gamma},1), ([\circ \alpha]_{\Gamma},1))$. Then, it is easy to see that $v_\Gamma^*$  is a valuation over the Nmatrix $\mathcal{M}_\Gamma^{\ecpl}$ such that $v_\Gamma^*(\alpha) \in D_{\mathcal{B}_{(\mathcal{A}_\Gamma)^*}^{\ecpl}}$ iff $\Gamma \vdash_{\ecpl} \alpha$, for every $\alpha$. Hence, $v_\Gamma^*(\gamma) \in D_{\mathcal{B}_{(\mathcal{A}_\Gamma)^*}^{\ecpl}}$ for every $\gamma \in \Gamma$, but  $v_\Gamma^*(\varphi) \not\in D_{\mathcal{B}_{(\mathcal{A}_\Gamma)^*}^{\ecpl}}$. From this $\Gamma\not\models_{Mat(\Ksw)} \varphi$, by Definition~\ref{conseqclassNmat}.
\end{proof}

\

\section{Swap structures for \mbc} \label{sect-sw-mbc}

A special subclass of \Ksw\ is formed by the swap structures for \mbc, defined as follows:

\begin{definition} \label{univ-sw-mbc}
The universe of swap structures  for \mbc\ over a Boolean algebra $\mathcal{A}$ is the set $\textsc{B}_\mathcal{A}^\mbc=\{z \in A^3 \ : \ z_1 \vee z_2=1 \ \mbox{ and } \ z_1 \wedge z_2 \wedge z_3=0  \}$.
\end{definition}

\begin{definition} \label{Swap-mbc}
Let $\mathcal{A}$ be a Boolean algebra.
A swap structure for \ecpl\ over $\mathcal{A}$ is said to be a
{\em swap structure for \mbc\ over $\mathcal{A}$} if its domain is included in $\textsc{B}_\mathcal{A}^\mbc$. 
Let $\Kmbc=\{{\cal B} \in \Ksw \ : \ {\cal B} \  \mbox{ is a swap structure for \mbc}\}$ be the class of swap structures for \mbc.
\end{definition}

If $\mathcal{M}$ is an Nmatrix and ({\bf ax}) is an axiom schema over the same signature, we say that  $\mathcal{M}$ {\em validates} ({\bf ax}) whenever $\models_{\mathcal{M}} \gamma$ for every instance $\gamma$ of ({\bf ax}). Then:

\begin{proposition} \label{CarAxKmbc}
$\Kmbc = \{ \mathcal{B} \in \Ksw \ : \  \  
\mbox{$\mathcal{M}(\mathcal{B})$ validates (\axtnd) and  (\axexp)} \}$.
\end{proposition}
\begin{proof}
Let  $\mathcal{B}$ be a swap structure for \mbc, and let $v$ be a valuation over $\mathcal{B}$. By definition of $\textsc{B}_\mathcal{A}^\mbc$ it follows that $\pi_1(v(\alpha)) \vee \pi_2(v(\alpha))=1$ and $\pi_1(v(\alpha)) \wedge \pi_2(v(\alpha)) \wedge \pi_3(v(\alpha))=0$. 
Let $\gamma=\alpha \vee \neg\alpha$ and $\gamma'=\cons\alpha \to (\alpha \to(\neg \alpha \to \beta))$ be  instances of axioms (\axtnd) and  (\axexp), respectively. By Definition~\ref{valNmat1} it follows that $\pi_1(v(\cons\alpha))=\pi_3(v(\alpha))$ and $\pi_1(v(\neg\alpha))=\pi_2(v(\alpha))$. Hence $\pi_1(v(\gamma)) = \pi_1(v(\alpha)) \vee  \pi_1(v(\neg\alpha))=\pi_1(v(\alpha)) \vee \pi_2(v(\alpha))=1$, obtaining so that $\mathcal{B}$ validates (\axtnd). On the other hand, $\pi_1(v(\gamma')) =  \pi_3(v(\alpha)) \to (\pi_1(v(\alpha)) \to (\pi_2(v(\alpha)) \to \pi_1(v(\beta))))=1$, since  $\pi_1(v(\alpha)) \wedge \pi_2(v(\alpha)) \wedge \pi_3(v(\alpha))=0$. This means that $\mathcal{B}$ validates (\axexp).

Conversely, let $\mathcal{B} \in \Ksw$ such that $\mathcal{M}(\mathcal{B})$ validates (\axtnd) and  (\axexp), and let $p$ and $q$ be two different propositional variables. Let $z \in |\mathcal{B}|$, and  consider a valuation $v$ over  $\mathcal{B}$ such that $v(p)=z$ and $\pi_1(v(q))=0$ (this is always possible since, by Definition~\ref{Swap-str}, $0 \in \pi_1[|\mathcal{B}|]$). Then $v(\neg p) \in \{w \in |\mathcal{B}| \ : \ w_1 = \pi_2(v(p))\} = \{w \in |\mathcal{B}| \ : \ w_1 = z_2\}$ and so  $v(p \lor\neg p) \in \{ u \in |\mathcal{B}| \ : \ u_1 = z_1 \vee \pi_1(v(\neg p))\} = \{ u \in |\mathcal{B}| \ : \ u_1 = z_1 \vee z_2 \}$.
But $v(p \lor\neg p) \in D_\mathcal{B}$, by hypothesis, then $\pi_1(v(p \lor\neg p))=z_1 \lor z_2=1$. On the other hand $v(\cons p \to (p \to (\neg p \to q))) \in D_\mathcal{B}$, since by  hypothesis  $\mathcal{B}$ validates (\axexp). Hence, $\pi_1(v(\cons p \to (p \to (\neg p \to q))))=1$. From this, and reasoning as above, $\pi_3(v(p)) \to (\pi_1(v(p)) \to (\pi_2(v(p)) \to 0))=1$. This means that $\pi_1(v(p)) \land \pi_2(v(p)) \land \pi_3(v(p))=0$, that is, $z_1 \land z_2 \land z_3=0$. Therefore $|\mathcal{B}| \subseteq \textsc{B}_\mathcal{A}^\mbc$, whence $\mathcal{B} \in \Kmbc$, by Definition~\ref{Swap-mbc}. 
\end{proof}

\begin{definition}
The full subcategory in  \sw\ of swap structures for \mbc\ will be denoted by \swmbc. 
\end{definition}

Clearly,  \swmbc\ is a full subcategory in  $\malg(\Sigma)$.
Thus,  the class of objects of \swmbc\  is $\Kmbc$, and the morphisms between two given swap structures for \mbc\ are the homomorphisms between them, seeing as multialgebras over $\Sigma$.

\begin{definition} \label{maxBmbC}
Let  $\mathcal{A}$ be a Boolean algebra. The {\em full  swap structure  for \mbc\ over  $\mathcal{A}$}, denoted by $\mathcal{B}_\mathcal{A}^\mbc$, is the unique swap structure  for \mbc\ with domain $\textsc{B}_\mathcal{A}^\mbc$ such that, for every $z$ and $w$ in $\textsc{B}_\mathcal{A}^\mbc$:

\begin{itemize}
 \item[(i)] $z\#w = \{u\in \textsc{B}_\mathcal{A}^\mbc \ : \ u_1=z_1\#w_1\}$, for each $\#\in \{\wedge,\vee, \to\}$;
\item[(ii)] $\neg (z) = \{u\in \textsc{B}_\mathcal{A}^\mbc \ : \ u_1=z_2\}$;
 \item[(iii)] $\circ (z) = \{u\in \textsc{B}_\mathcal{A}^\mbc \ : \ u_1=z_3\}$.
\end{itemize}
\end{definition}

Let $\{\mathcal{A}_i \ : \ i \in I\}$ be a family of Boolean algebras such that $I \neq \emptyset$, and $\mathcal{A}_i = \langle A_i, \wedge_i, \vee_i, \to_i, 0_i, 1_i \rangle$ for every $i \in I$.
Let $A=\prod_{i \in I} A_i$ be the standard construction of the cartesian product of the family of sets $\{A_i \ : \ i \in I\}$ with canonical projections $\pi_i:A \to A_i$ for every $i \in I$. Let $\mathcal{A}$ be the algebra with domain $A$ such that, for every $a,b \in A$ and $\# \in \{\wedge,\vee, \to\}$, $a \# b \in A$ is given by $(a \# b)(i) = a(i) \#_i b(i)$, for every $i \in I$. Let $0_\mathcal{A}, 1_\mathcal{A} \in A$ such that  $0_\mathcal{A}(i) = 0_i$ and $1_\mathcal{A}(i) = 1_i$, for every $i \in I$.  It is well known that $\mathcal{A} = \langle A, \wedge, \vee, \to, 0, 1 \rangle$ is a Boolean algebra where the canonical projections $\pi_i:A \to A_i$ are homomorphisms of  Boolean algebras such  that $\langle \mathcal{A}, \{\pi_i \ : \ i \in I\}\rangle$ is the product of the family $\{\mathcal{A}_i \ : \ i \in I\}$ in the category of Boolean algebras. The Boolean algebra $\mathcal{A}$ will be denoted by  $\prod_{i \in I} \mathcal{A}_i$. The case for $I=\emptyset$ is obvious, producing the one element Boolean algebra.

Consider again a family $\mathcal{F} = \{\mathcal{A}_i \ : \ i \in I\}$ of Boolean algebras such that $I \neq \emptyset$, and let $\mathcal{A}=\prod_{i \in I} \mathcal{A}_i$ be its product in the category of Boolean algebras, as described above. We want to show that the product $\mathcal{B}=\prod_{i \in I}\mathcal{B}_{\mathcal{A}_i}^\mbc$  in  $\malg(\Sigma)$ (recall Proposition~\ref{prods})  of the family of multialgebras $\{\mathcal{B}_{\mathcal{A}_i}^\mbc \ : \ i \in I\}$ is isomorphic  in  $\malg(\Sigma)$ (recall Proposition~\ref{isos}) to the multialgebra $\mathcal{B}_\mathcal{A}^\mbc$ (recall Definition~\ref{maxBmbC}).

To begin with, some notation is required. Let $\pi^i_{(j)}:(A_i)^3 \to A_i$ be the canonical projections, for $i \in I$ and $1 \leq j \leq 3$. Observe that, if $a \in |\mathcal{B}|=\prod_{i \in I} \textsc{B}_{\mathcal{A}_i}^\mbc$ and $i \in I$ then $a(i) \in \textsc{B}_{\mathcal{A}_i}^\mbc \subseteq (A_i)^3$. Thus, for every $1 \leq j \leq 3$ let  $z_j \in \prod_{i \in I} A_i$ such that, for every $i \in I$, $z_j(i)=\pi^i_{(j)}(a(i))$. Then $z=(z_1,z_2,z_3)$ belongs to  $|\mathcal{A}|^3$. Moreover, it can be proven that $z$ belongs to  $\textsc{B}_\mathcal{A}^\mbc$. Indeed, for every $i \in I$, $z_1(i) \vee_i z_2(i) = \pi^i_{(1)}(a(i)) \vee_i  \pi^i_{(2)}(a(i)) = 1_i$ since $a(i) \in \textsc{B}_{\mathcal{A}_i}^\mbc$. From this, $z_1 \vee z_2 = 1_\mathcal{A}$. Analogously it can be proven that $z_1 \wedge z_2 \wedge z_3 = 0_\mathcal{A}$. 

This allows to define a mapping $f_\mathcal{F}:\prod_{i \in I} \textsc{B}_{\mathcal{A}_i}^\mbc \to \textsc{B}_{\prod_{i \in I} \mathcal{A}_i}^\mbc$ such that, for every  $a \in \prod_{i \in I} \textsc{B}_{\mathcal{A}_i}^\mbc$, $f_\mathcal{F}(a)=z$ where $z=(z_1,z_2,z_3)$ is defined as above.

\begin{proposition} \label{isofam}
Let $\mathcal{F} = \{\mathcal{A}_i \ : \ i \in I\}$ be  a family of Boolean algebras such that $I \neq \emptyset$. Then, the mapping $f_\mathcal{F}:\prod_{i \in I} \textsc{B}_{\mathcal{A}_i}^\mbc \to \textsc{B}_{\prod_{i \in I} \mathcal{A}_i}^\mbc$ is an isomorphism in $\malg(\Sigma)$.
\end{proposition}
\begin{proof}
Clearly $f_\mathcal{F}$ is a bijective mapping such that its inverse mapping is given by  $f_\mathcal{F}^{-1}:\textsc{B}_{\prod_{i \in I} \mathcal{A}_i}^\mbc \to \prod_{i \in I} \textsc{B}_{\mathcal{A}_i}^\mbc$ where  $f_\mathcal{F}^{-1}(z_1,z_2,z_3) = a$, with $a(i)=(z_1(i),z_2(i),z_3(i))$ for every $i \in I$. It is also clear that, for every $a,b \in \prod_{i \in I} \textsc{B}_{\mathcal{A}_i}^\mbc$ and $\#\in \{\wedge,\vee, \to\}$:
\begin{itemize}
 \item[(i)] $f_\mathcal{F}[a\#b] = f_\mathcal{F}(a) \# f_\mathcal{F}(b)$;
\item[(ii)] $f_\mathcal{F}[\neg a] = \neg f_\mathcal{F}(a)$; and
 \item[(iii)] $f_\mathcal{F}[\circ a] = \circ f_\mathcal{F}(a)$
\end{itemize}
(the details are left to the reader). The result follows from Proposition \ref{isos}.
\end{proof}

\begin{proposition} \label{prods-sw}
The category \swmbc\  has arbitrary products.
\end{proposition}
\begin{proof} Let $\mathcal{F} = \{\mathcal{B}_i \ : \ i \in I\}$ be a family of swap structures for \mbc, and assume that $I \neq \emptyset$ (the case $I=\emptyset$ is trivial). By definition of $\Kmbc$, for each $i \in I$ there is a Boolean algebra $\mathcal{A}_i $ such that $\mathcal{B}_i  \subseteq \mathcal{B}_{\mathcal{A}_i}^\mbc$. Since \swmbc\ is a subcategory of $\malg(\Sigma)$ (where $\Sigma$ is the signature of \mbc), and the latter has arbitrary products (cf. Proposition \ref{prods}), there exists the product $\langle \mathcal{B}, \{\pi_i \ : \ i \in I\}\rangle$ of $\mathcal{F}$ in $\malg(\Sigma)$. By the proof of  Proposition~\ref{prods}, it is possible to define $\mathcal{B}$ in such a way that $\mathcal{B} \subseteq \prod_{i \in I} \mathcal{B}_{\mathcal{A}_i}^\mbc$, where the multialgebra $\prod_{i \in I} \mathcal{B}_{\mathcal{A}_i}^\mbc$ is also constructed as in the proof of Proposition~\ref{prods}. Let $h:\mathcal{B} \to \prod_{i \in I} \mathcal{B}_{\mathcal{A}_i}^\mbc$ be the inclusion homomorphism. Now, let $\mathcal{G} = \{\mathcal{A}_i \ : \ i \in I\}$ and let $f_\mathcal{G}:\prod_{i \in I} \mathcal{B}_{\mathcal{A}_i}^\mbc \to \mathcal{B}_{\prod_{i \in I} \mathcal{A}_i}^\mbc$ be the isomorphism in $\malg(\Sigma)$ of Proposition~\ref{isofam}. Then, the  homomorphism $f_\mathcal{G} \circ h:\mathcal{B} \to \mathcal{B}_{\prod_{i \in I} \mathcal{A}_i}^\mbc$ is an injective function
$$\xymatrix{\mathcal{B} \ar@{^{(}->}[rr]^{\hspace*{-8mm}h} \ar@{_{(}->}[rrd]_{f_\mathcal{G} \circ h} && \prod_{i \in I} \mathcal{B}_{\mathcal{A}_i}^\mbc \ar[d]^{f_\mathcal{G}}\\
&& \mathcal{B}_{\prod_{i \in I} \mathcal{A}_i}^\mbc}$$
and so it induces an isomorphism  $\overline{f_\mathcal{G} \circ h}$ in $\malg(\Sigma)$ between $\mathcal{B}$ and the submultialgebra $\mathcal{B}'=(f_\mathcal{G} \circ h)(\mathcal{B})$ of $\mathcal{B}_{\prod_{i \in I} \mathcal{A}_i}^\mbc$, by Proposition~\ref{epi-mono}. 
This means that $\langle \mathcal{B}', \{\pi_i \circ (\overline{f_\mathcal{G} \circ h})^{-1} \ : \ i \in I\}\rangle$ is another realization of the product of $\mathcal{F}$ in $\malg(\Sigma)$. 
$$\xymatrix{&\mathcal{B}  \ar[ld]_{\pi_i}\ar@{^{(}->}[rr]^{\hspace*{-8mm}f_\mathcal{G} \circ h} && \mathcal{B}_{\prod_{i \in I} \mathcal{A}_i}^\mbc\\
\mathcal{B}_i&&& \mathcal{B}' \ar[llu]^{(\overline{f_\mathcal{G} \circ h})^{-1}} \ar@{^{(}->}[u] \ar[lll]^{\pi_i \circ (\overline{f_\mathcal{G} \circ h})^{-1}} }$$
Given that  \swmbc\ is a full subcategory of $\malg(\Sigma)$ and by observing that $\mathcal{B}'$ is an object of  \swmbc, it follows that    $\langle \mathcal{B}', \{\pi_i \circ (\overline{f_\mathcal{G} \circ h})^{-1} \ : \ i \in I\}\rangle$ is a construction for the product  in \swmbc\ of the family $\mathcal{F}$.
\end{proof}

\

Let \balg\ be the category of Boolean algebras defined over signature $\Sigma_{\rm BA}=\{\land, \lor, \to,0,1\}$, with Boolean algebras homomorphisms as their morphisms. 
Then, the assignment $\mathcal{A} \in \balg \ \mapsto \ \mathcal{B}_\mathcal{A}^\mbc \in \swmbc$ is functorial, as it will be stated in Corollary~\ref{functor} below.

\begin{proposition} Let $f:\mathcal{A} \to \mathcal{A}'$ be a homomorphism between Boolean algebras. Then it induces a homomorphism  $f_\ast:\mathcal{B}_\mathcal{A}^\mbc \to \mathcal{B}_{\mathcal{A}'}^\mbc$ of multialgebras given by $f_\ast(z) = (f(z_1),f(z_2),f(z_3))$. Moreover, $(f \circ g)_\ast = f_\ast \circ g_\ast$ and $(id_\mathcal{A})_\ast = id_{\mathcal{B}_\mathcal{A}^\mbc}$, where $id_\mathcal{A}:\mathcal{A} \to \mathcal{A}$ and $id_{\mathcal{B}_\mathcal{A}^\mbc}:\mathcal{B}_\mathcal{A}^\mbc \to \mathcal{B}_\mathcal{A}^\mbc$ are the corresponding identity homomorphisms.
\end{proposition}
\begin{proof}
Given a homomorphism $f:\mathcal{A} \to \mathcal{A}'$ between Boolean algebras, let  $f_\ast:\textsc{B}_\mathcal{A}^\mbc \to \textsc{B}_{\mathcal{A}'}^\mbc$ be the mapping such that $f_\ast(z) = (f(z_1),f(z_2),f(z_3))$ for every $z \in \textsc{B}_\mathcal{A}^\mbc$. If $z,w \in \textsc{B}_\mathcal{A}^\mbc$ and $\#\in \{\wedge,\vee, \to\}$ then, for every $u \in (z \# w)$, $u_1=z_1 \# w_1$ and so $f(u_1)=f(z_1) \# f(w_1)$. That is, $(f_\ast(u))_1 = (f_\ast(z))_1 \# (f_\ast(w))_1$. This means that $f_\ast[z \# w] = \{f_\ast(u) \ : \ u \in  (z \# w)\} \subseteq \{u' \in   \textsc{B}_{\mathcal{A}'}^\mbc \ : \ u'_1 = (f_\ast(z))_1 \# (f_\ast(w))_1\} = f_\ast(z) \# f_\ast(w)$. On the other hand, if $z \in \textsc{B}_\mathcal{A}^\mbc$ and $u \in \neg z$ then $u_1=z_2$ whence $(f_\ast(u))_1=f(u_1)=f(z_2)=(f_\ast(z))_2$. This means that $f_\ast(u) \in \{u' \in \textsc{B}_{\mathcal{A}'}^\mbc \ : \ u'_1 = (f_\ast(z))_2\} = \neg f_\ast(z)$ and so $f_\ast[\neg z] \subseteq  \neg f_\ast(z)$. Analogously it can be proven that $f_\ast[\circ z] \subseteq  \circ f_\ast(z)$. This shows that $f_\ast$ is indeed a homomorphism $f_\ast:\mathcal{B}_\mathcal{A}^\mbc \to \mathcal{B}_{\mathcal{A}'}^\mbc$ in \swmbc. The rest of the proof is immediate, by the very definition of $f_\ast$.
\end{proof}

\begin{corollary} \label{functor} There exists a functor $K^*_\mbc:\balg \to \swmbc$ given by $K^*_\mbc(\mathcal{A})=\mathcal{B}_\mathcal{A}^\mbc$ for every Boolean algebra $\mathcal{A}$, and $K^*_\mbc(f)=f_\ast$ for every  homomorphism $f:\mathcal{A} \to \mathcal{A}'$ in \balg.
\end{corollary}

\begin{definition} \label{KalmanFunc}
The functor $K^*_\mbc:\balg \to \swmbc$ of Corollary~\ref{functor} is called {\em dual Kalman's functor for \swmbc}.
\end{definition}

\begin{remark} [Kalman's construction and twist structures] \label{obs-Kalman}
The name {\em dual Kalman's functor} was used in Definition~\ref{KalmanFunc} because of the analogy with a construction proposed in 1958 by J. Kalman (see~\cite{kal:58}). This point will be clarified in sections~\ref{Kalman-twist} and~\ref{swap=twist}.   
\end{remark}

\begin{proposition} \label{preservprods}
The dual Kalman's functor $K^*_\mbc:\balg \to \swmbc$  preserves arbitrary products.
\end{proposition}
\begin{proof}
It is an immediate consequence of Proposition \ref{isofam} and the fact that \swmbc\ is a full subcategory of $\malg(\Sigma)$.
\end{proof}

\begin{proposition} \label{preservsubs}
The  dual Kalman's functor $K^*_\mbc:\balg \to \swmbc$  preserves subalgebras in the following sense: if $\mathcal{A}$ is  a subalgebra de $\mathcal{A}'$ in the category of Boolean algebras, then $\mathcal{B}_\mathcal{A}^\mbc  \subseteq \mathcal{B}_{\mathcal{A}'}^\mbc$ according to Definition~\ref{subal}.
\end{proposition}
\begin{proof}
It is an immediate consequence of the definitions.
\end{proof}

\

Moreover, the following holds:

\begin{proposition} \label{preservmonos}
The  dual Kalman's functor $K^*_\mbc:\balg \to \swmbc$  preserves monomorphisms.
\end{proposition}
\begin{proof}
Let $f:\mathcal{A} \to \mathcal{A}'$ be a monomomorphism between Boolean algebras, and let $f_\ast:\mathcal{B}_\mathcal{A}^\mbc \to \mathcal{B}_{\mathcal{A}'}^\mbc$ be the induced homomorphism of multialgebras given by $f_\ast(z) = (f(z_1),f(z_2),f(z_3))$. It is well-known that every monomorphism in \balg\ is an injective function, and then $f$ is injective. From this it is immediate to see that $f_\ast$ is also an injective function. As a consequence of Proposition~\ref {monos}, $f_\ast$ is a monomorphism in the category $\malg(\Sigma)$. Given that \swmbc\ is a full subcategory of $\malg(\Sigma)$, it follows that $f_\ast$ is a monomorphism in the category \swmbc.
\end{proof}

\section{Swap structures semantics for \mbc} \label{swap-sem-mbc}

As it was done in Definition~\ref{Nmatrix},  each $\mathcal{B} \in \Kmbc$ induces naturally a non-deterministic matrix  $\mathcal{M}(\mathcal{B})=(\mathcal{B}, D_\mathcal{B})$. Moreover,  in~\cite[Theorem 6.4.8]{CC16} it was proven that the class $Mat(\Kmbc) = \{\mathcal{M}(\mathcal{B}) \ : \ \mathcal{B} \in  \Kmbc\}$ semantically characterizes \mbc, by considering the consequence relation $\models_{Mat(\Kmbc)}$ as in Definition~\ref{conseqclassNmat}. However, the proof given in~\cite{CC16} is indirect: it lies on the equivalence between the swap-structures semantics and the  Fidel structures  semantics for \mbc, together with the adequacy of \mbc\ \mbox{w.r.t.} the latter structures. Now, a direct proof of the adequacy of \mbc\ w.r.t. swap structures will be given (recalling the consequence relation introduced in Definition~\ref{conseqclassNmat}).

\begin{theorem}  [Adequacy of \mbc\ w.r.t. swap structures] \label{adeq-mbC}
Let $\Gamma \cup \{\varphi\} \subseteq For(\Sigma)$  be a set of formulas. Then:
$\Gamma \vdash_\mbc \varphi$ \ iff \  $\Gamma\models_{Mat(\Kmbc)} \varphi$.
\end{theorem}
\begin{proof}
The proof is similar to that for Theorem~\ref{adeq-CPLP1}.\\
`Only if' part (Soundness): Assume that $\Gamma \vdash_\mbc \varphi$. Let  $\mathcal{B}$ be a swap structure for \mbc, and let  $v$ be a valuation over $\mathcal{B}$ such that $v(\gamma) \in D_\mathcal{B}$ for every $\gamma \in \Gamma$. 
By Theorem~\ref{adeq-CPLP1}, $v$ validates every axiom of \cplp. On the other hand, $v$ also validates (\axtnd) and  (\axexp), by Proposition~\ref{CarAxKmbc}. 
In addition,  $v(\beta) \in D_\mathcal{B}$ whenever $v(\alpha)  \in D_\mathcal{B}$ and $v(\alpha \to\beta) \in D_\mathcal{B}$, and so trueness in $v$ is preserved by~(\MP). Hence,  it follows that $v(\varphi) \in D_\mathcal{B}$. This shows that   $\Gamma\models_{Mat(\Kmbc)} \varphi$.\\ [2mm]
`If' part (Completeness): Assume that $\Gamma \nvdash_{\mbc} \varphi$. Define in $For(\Sigma)$ the following relation: $\alpha \equiv_\Gamma \beta$ iff $\Gamma \vdash_{\mbc} \alpha \imp \beta$ and $\Gamma \vdash_{\mbc} \beta \imp \alpha$. As in the proof of  Theorem~\ref{adeq-CPLP1} it follows that $\equiv_\Gamma$ is an equivalence relation such that  the quotient set $A_\Gamma \defin For(\Sigma)/_{\equiv_\Gamma}$ is a \algcpl, where $[\alpha]_{\Gamma} \,\#\, [\beta]_{\Gamma} \defin [\alpha \# \beta]_{\Gamma}$, for $\# \in \{\land,\lor,\imp\}$. Moreover, $0_\Gamma \defin [p_1 \wedge \neg p_1 \wedge \cons p_1]_\Gamma$ and $1_\Gamma \defin [p_1 \to p_1]_\Gamma$ are the bottom and top elements of $A_\Gamma$, respectively, and so $A_\Gamma$ is the domain of a Boolean algebra $\mathcal{A}_\Gamma$, by Proposition~\ref{BA-HA}(2). Let  $\mathcal{B}_{\mathcal{A}_\Gamma}^{\mbc}$ be the corresponding full swap structure for \mbc\ (recall Definition~\ref{maxBmbC}), and let $\mathcal{M}_\Gamma^{\mbc} \defin\mathcal{M}(\mathcal{B}_{\mathcal{A}_\Gamma}^{\mbc})$. The mapping $v_\Gamma:For(\Sigma) \to  \textsc{B}_{\mathcal{A}_\Gamma}^\mbc$ given by $v_\Gamma(\alpha) =([\alpha]_{\Gamma},[\neg \alpha]_{\Gamma}, [\circ \alpha]_{\Gamma})$  is a valuation over the Nmatrix $\mathcal{M}_\Gamma^{\mbc}$ such that $v_\Gamma(\alpha) \in D_{\mathcal{B}_{\mathcal{A}_\Gamma}^{\mbc}}$ iff $\Gamma \vdash_{\mbc} \alpha$, for every $\alpha$. From this, $v_\Gamma[\Gamma] \subseteq D_{\mathcal{B}_{\mathcal{A}_\Gamma}^{\mbc}}$ but  $v_\Gamma(\varphi) \not\in D_{\mathcal{B}_{\mathcal{A}_\Gamma}^{\mbc}}$. Therefore $\Gamma\not\models_{Mat(\Kmbc)} \varphi$, by Definition~\ref{conseqclassNmat}.
\end{proof}

The Nmatrix $\mathcal{M}_5^\mbc=\mathcal{M}\big(\mathcal{B}_{\mathbb{A}_2}^\mbc\big)$ induced by the full swap structure $\mathcal{B}_{\mathbb{A}_2}^\mbc$ defined over the two-element Boolean algebra  $\mathbb{A}_2$  (see Definition~\ref{maxBmbC}) was originally introduced by A. Avron in  \cite{avr:05},  in order to semantically characterize  the logic \mbc. The domain of the multialgebra $\mathcal{B}_{\mathbb{A}_2}^\mbc$ is the set  $\textsc{B}_{\mathbb{A}_2}^\mbc = \big\{T, \, t, \, t_0, \, F, \, f_0\big\}$ such that $T=(1,0,1)$, $t=(1,1,0)$, $t_0=(1,0,0)$, $F=  (0,1,1)$, and $f_0=(0,1,0)$. 
Let D be the set of designated elements of the Nmatrix $\mathcal{M}_5^\mbc$. Then, $\textrm{D}=\{T, \, t, \, t_0\}$. Let $\textrm{ND}=\big\{F, \, f_0\big\}$ be the set of non-designated truth-values. 
The multioperations proposed by Avron over the set $\textsc{B}_{\mathbb{A}_2}^\mbc$ corresponds exactly with  that for $\mathcal{B}_{\mathbb{A}_2}^\mbc$ described in Definition~\ref{maxBmbC}. Namely,

\

\begin{center}
\begin{tabular}{|c|c|c|c|c|c|}
\hline
 $\wedge^\matM$ & $T$   & $t$ & $t_0$ & $F$ & $f_0$ \\
 \hline \hline
    $T$    & D  & D & D & ND & ND   \\ \hline
     $t$    & D  & D & D & ND & ND  \\ \hline
     $t_0$    & D  & D & D & ND & ND  \\ \hline
     $F$    & ND  & ND & ND & ND & ND  \\ \hline
     $f_0$    & ND  & ND & ND & ND & ND  \\ \hline
\end{tabular}
\hspace{0.5cm}
\begin{tabular}{|c|c|c|c|c|c|}
\hline
 $\vee^\matM$ & $T$   & $t$ & $t_0$ & $F$ & $f_0$ \\
 \hline \hline
    $T$    & D  & D & D & D & D   \\ \hline
     $t$    & D  & D & D & D & D  \\ \hline
     $t_0$    & D  & D & D & D & D  \\ \hline
     $F$    & D  & D & D & ND & ND  \\ \hline
     $f_0$    & D  & D & D & ND & ND  \\ \hline
\end{tabular}
\end{center}

\

\begin{center}
\begin{tabular}{|c|c|c|c|c|c|}
\hline
 $\to^\matM$ & $T$   & $t$ & $t_0$ & $F$ & $f_0$ \\
 \hline \hline
    $T$    & D  & D & D & ND & ND   \\ \hline
     $t$    & D  & D & D & ND & ND  \\ \hline
     $t_0$    & D  & D & D & ND & ND  \\ \hline
     $F$    & D  & D & D & D & D  \\ \hline
     $f_0$    & D  & D & D & D & D  \\ \hline
\end{tabular}
\hspace{0.5cm}
\begin{tabular}{|c||c|} \hline
$\quad$ & $\neg^\matM$ \\
 \hline \hline
    $T$   & ND    \\ \hline
     $t$   & D    \\ \hline
     $t_0$   & ND    \\ \hline
     $F$   & D    \\ \hline
     $f_0$   & D    \\ \hline
\end{tabular}
\hspace{0.5cm}
\begin{tabular}{|c||c|}
\hline
 $\quad$ & $\circ^\matM$ \\
 \hline \hline
    $T$   & D    \\ \hline
     $t$   & ND    \\ \hline
     $t_0$   & ND    \\ \hline
     $F$   & D    \\ \hline
     $f_0$   & ND    \\ \hline
\end{tabular}
\end{center}

\ \\

\noindent  It was  proved in~\cite{avr:05} that \mbc\ is adequate for $\mathcal{M}_5^\mbc$:

\begin{theorem} \label{comp-mbc-full}
For every set of formulas $\Gamma \cup \{\varphi\} \subseteq For(\Sigma)$: 
$\Gamma \vdash_\mbc \varphi$ \ iff \  $\Gamma\models_{\mathcal{M}_5^\mbc} \varphi$.
\end{theorem}

A new proof of the latter result was obtained in~\cite[Corollary~6.4.10]{CC16}, by relating bivaluations for \mbc\  with the Nmatrix $\mathcal{M}\big(\mathcal{B}_{\mathbb{A}_2}^\mbc\big)$.

\begin{definition} [\cite{CCM}] \label{bivalold}
A function $\mu:For(\Sigma)\to \big\{0,1\big\}$ is a {\em bivaluation for \mbc}
if it  satisfies the   following clauses:\\[2mm]
{\bf (\vale)} \ $\mu(\alpha \land \beta) = 1$  \ iff \  $\mu(\alpha) = 1$ \ and \
$\mu(\beta) = 1$ \\[2mm]
{\bf (\valou)} \ $\mu(\alpha \lor \beta) = 1$ \  iff \ $\mu(\alpha) = 1$ \ or \
$\mu(\beta) = 1$ \\[2mm]
{\bf (\valimp)} \ $\mu(\alpha \to \beta) = 1$ \  iff \  $\mu(\alpha) = 0$ \ or \
$\mu(\beta) = 1$ \\[2mm]
{\bf (\valnot)} \ $\mu(\lnot \alpha)=0$ \ implies  \ $\mu(\alpha)=1$ \\[2mm]
{\bf (\valbola)} \ $\mu(\cons \alpha) = 1$ \ implies \ $\mu(\alpha)=0$ \ or \ $\mu(\lnot \alpha)=0$. \\[2mm]
The consequence relation of \mbc\ w.r.t. bivaluations is defined as follows: for every set of formulas $\Gamma \cup \{\varphi\} \subseteq For(\Sigma)$, $\Gamma \models_\mbc^2 \varphi$ \ iff \ $\mu(\varphi)=1$ for every bivaluation for \mbc\ such that $\mu[\Gamma] \subseteq \{1\}$.
\end{definition}

\begin{theorem} [\cite{CCM}]  \label{comp-bival-mbC}
For every set of formulas $\Gamma \cup \{\varphi\} \subseteq For(\Sigma)$: 
$\Gamma \vdash_\mbc \varphi$ \ iff \  $\Gamma\models_{\mbc}^2 \varphi$.
\end{theorem}

\begin{definition} [\cite{CC16}]  \label{val-bival-mbC}
Let $\mu$ be a bivaluation for \mbc. The valuation over  the Nmatrix $\mathcal{M}\big(\mathcal{B}_{\mathbb{A}_2}^\mbc\big)$ induced by $\mu$ is given by $v_\mu^\mbc(\alpha) \defin (\mu(\alpha),\mu(\neg\alpha),\mu(\cons\alpha))$ for every formula $\alpha$.
\end{definition}

By showing that $v_\mu^\mbc$ is indeed a valuation over $\mathcal{M}\big(\mathcal{B}_{\mathbb{A}_2}^\mbc\big)$ such that $v_\mu^\mbc(\alpha) \in D$ iff $\mu(\alpha)=1$, Theorem~\ref{comp-mbc-full} follows easily (see~\cite[Corollary~6.4.10]{CC16}). 

As observed in \cite[Chapter 6]{CC16}, Avron's result means that the Nmatrix induced by the full swap structure $\mathcal{B}_{\mathbb{A}_2}^\mbc$ defined over the two-element Boolean algebra  $\mathbb{A}_2$ is sufficient for characterizing the logic \mbc, and so it represents, in a certain way, the whole class $\Kmbc$  of swap structures for \mbc.
One interesting question is to prove that the 5-element multialgebra $\mathcal{B}_{\mathbb{A}_2}^\mbc$ generates (in some sense) the class $\Kmbc$, in analogy to the fact that the 2-element Boolean algebra  $\mathbb{A}_2$ generates the class of Boolean algebras.

Indeed, in~\cite{Birk:35}  G. Birkhoff proves that, for  every Boolean algebra $\mathcal{A}$, there exists a set $I$ and a monomorphism  of Boolean algebras $h:\mathcal{A} \to \prod_{i \in I}\mathbb{A}_2$. Moreover, in 1944 he obtained the nowadays known as  {\em Birkhoff's representation theorem}, which states that if $\mathbb{K}$ is an equationally defined class of algebras then every algebra in the class is a subdirect product of subdirectly irreducible algebras of $\mathbb{K}$ (see~\cite{GB2}). The generalization of this theorem to multialgebras is an open problem (see Section~\ref{conclusion}). From  the representation theorem for Boolean algebras~\cite{Birk:35}, and taking into account the properties of the dual Kalman's functor $K^*_\mbc:\balg \to \swmbc$, a representation theorem for the class $\Kmbc$  of swap structures for \mbc\ can be obtained:

\begin{theorem} [Representation Theorem for $\Kmbc$] \label{teo-repr-Kmbc}  Let $\mathcal{B}$ be a swap structure  for \mbc. Then, there exists a set $I$ and a monomorphism  of multialgebras $\hat h:\mathcal{B} \to \prod_{i \in I}\mathcal{B}_{\mathbb{A}_2}^\mbc$.
\end{theorem}

\begin{proof}
 Let $\mathcal{B}$ be a swap structure for \mbc. Then, there is a Boolean algebra $\mathcal{A}$ such that $\mathcal{B} \subseteq \mathcal{B}_\mathcal{A}^\mbc$. Let $g:\mathcal{B} \to \mathcal{B}_\mathcal{A}^\mbc$ be the inclusion monomorphism in \swmbc. Using Birkhoff's  representation theorem for Boolean algebras, there exists a set $I$ and a monomorphism $h:\mathcal{A} \to \prod_{i \in I}\mathcal{A}'_i$  of Boolean algebras, where 
$\mathcal{A}'_i =\mathbb{A}_2$, for every $i \in I$. By Proposition~\ref{preservmonos}, there is a monomorphism $h_\ast:\mathcal{B}_\mathcal{A}^\mbc \to \mathcal{B}_{\prod_{i \in I} \mathcal{A}'_i}^\mbc$. Let 
$f_\mathcal{G}:\prod_{i \in I} \mathcal{B}_{\mathcal{A}'_i}^\mbc \to \mathcal{B}_{\prod_{i \in I} \mathcal{A}'_i}^\mbc$ be the isomorphism in $\malg(\Sigma)$ of Proposition~\ref{isofam}, where  $\mathcal{G} = \{\mathcal{A}'_i \ : \ i \in I\}$. By definition of $\mathcal{A}'_i$ it follows that $\mathcal{B}_{\mathcal{A}'_i}^\mbc = \mathcal{B}_{\mathbb{A}_2}^\mbc$, for every $i \in I$. Then $\hat h:\mathcal{B} \to \prod_{i \in I}\mathcal{B}_{\mathbb{A}_2}^\mbc$ is a monomorphism in   $\malg(\Sigma)$, where $\hat h=f_\mathcal{G}^{-1} \circ h_\ast \circ g$.
  \end{proof}

\begin{remark} \label{rem-subdir-irred}
 It is not clear whether the latter result is a representation theorem in the stronger sense of~\cite{GB2}. Indeed, the notion of  subdirectly irreducible multialgebras should be studied. After this, it should be proved that the factors $\mathcal{B}_{\mathbb{A}_2}^\mbc$ are indeed subdirectly irreducible in that sense.
\end{remark}

\

In universal algebra, a variety is an equationally defined class of algebras. It is equivalent to require that the class is  closed under products, subalgebras and homomorphic images.
From the previous result, and given that an equation theory for multialgebras is still incipient, it is natural to ask about the possibility of the class \Kmbc\ being  closed under products, submultialgebras and homomorphic images. We known that \Kmbc\ is closed  under products (by Proposition~\ref{prods-sw}) and submultialgebras (by the very definitions). Unfortunately, the class is not closed under homomorphic images.

Indeed, recall the notions of multicongruence (Definition~\ref{mcongru}), quotient multialgebra (Definition~\ref{quomultialg}) and the canonical map $p:A \to A/_\Theta$ for every multicongruence $\Theta$ (Proposition~\ref{canproj-quo-epi}). Now, let $\textrm{D}=\{z^1, z^2, z^3\}$ and $\textrm{ND}=\big\{z^4, z^5\big\}$ be an enumeration of the elements of the domain $\textsc{B}_{\mathbb{A}_2}^\mbc = \textrm{D} \cup \textrm{ND}$ of the  multialgebra $\mathcal{B}_{\mathbb{A}_2}^\mbc$. Let $\Theta$ be the equivalence relation asociated to the partition $\{a,b\}$ of $\textsc{B}_{\mathbb{A}_2}^\mbc$ such that $a=\{z^1,z^4\}$ and $b=\{z^2,z^3,z^5\}$. The relation  $\Theta$ has the following property: for every $z \in   \textrm{D}$ there exists some $w \in  \textrm{ND}$ such that $(z,w) \in \Theta$, and vice versa. From this, and by observing the definition of the multioperations in the  multialgebra $\mathcal{B}_{\mathbb{A}_2}^\mbc$, it follows that $\Theta$ is a  multicongruence over $\mathcal{B}_{\mathbb{A}_2}^\mbc$. It is easy to prove that the multioperations in the quotient multialgebra ${\mathcal{B}_{\mathbb{A}_2}^\mbc}/_\Theta$ are trivial, that is: for every $x,y \in  \{a,b\}$ and $\# \in \{\wedge,\vee,\to\}$, $(x \# y) = \neg x = \circ x = \{a,b\}$. Clearly ${\mathcal{B}_{\mathbb{A}_2}^\mbc}/_\Theta$  is not a swap structure for \mbc: otherwise, it would generate a trivial Nmatrix where the set of designated values is the whole domain. This would contradict \cite[Proposition 6.4.5(ii)]{CC16}, where it was proven that no Nmatrix in the class $Mat(\Kmbc)$ is trivial. This shows that ${\mathcal{B}_{\mathbb{A}_2}^\mbc}/_\Theta$, the homomorphic image of the canonical  map $p:\mathcal{B}_{\mathbb{A}_2}^\mbc \to {\mathcal{B}_{\mathbb{A}_2}^\mbc}/_\Theta$, does not belong to the class \Kmbc, despite its domain $\mathcal{B}_{\mathbb{A}_2}^\mbc$ is in \Kmbc.

We thus prove the following:

\begin{proposition}
The class \Kmbc\ of multialgebras is closed under submultialgebras and (direct) products, but it is not closed under homomorphic images.
\end{proposition}

\section{Swap structures for some extensions of \mbc} \label{swap-ext-mbc}

In  \cite[Chapter 6]{CC16} the concept of swap structure for \mbc\ was generalized to some axiomatic extensions of \mbc. As observed in the beginning of Section~\ref{sect-sw-mbc}, these structures will be reintroduced here in a slightly modified form, more suitable to an algebraic study of them.

\begin{definition} (\cite[Definition 3.1.1]{CC16})  \label{mbcciw}
The logic \mbcciw\ is obtained from \mbc\  by adding the axiom schema 
$$\circ\alpha \vee (\alpha \wedge \neg \alpha)       \hspace{2cm}     (\axwci)$$
\end{definition}

\begin{definition} \label{univ-sw-mbcciw}
Let $\mathcal{A}$ be a Boolean algebra. The universe of swap structures  for \mbcciw\ over $\mathcal{A}$ is the set $\textsc{B}_\mathcal{A}^\axwci=\{z \in \textsc{B}_\mathcal{A}^\mbc \ : \ z_3 \vee (z_1 \wedge z_2)=1\}$.
\end{definition}

\

Clearly, $\textsc{B}_\mathcal{A}^\axwci=\{z \in A^3 \ :  \ z_1 \vee z_2=1 \ \mbox{ and } \  z_3 =  {\sim} (z_1 \wedge z_2)\}$.\footnote{Recall that, in this paper, $\sim$ denotes the Boolean complement in a Bolean algebra.}

\begin{definition} \label{Swap-mbcciw}
Let $\mathcal{A}$ be a Boolean algebra.
A swap structure for \mbc\ over $\mathcal{A}$ is said to be a
{\em swap structure for \mbcciw\ over $\mathcal{A}$} if its domain is included in $\textsc{B}_\mathcal{A}^\axwci$.
\end{definition}

Let $\Kmbcciw=\{{\cal B} \in \Kmbc \ : \ {\cal B} \  \mbox{ is a swap structure for \mbcciw}\}$ be the class of swap structures for \mbcciw.  The following result justifies Definition~\ref{Swap-mbcciw}:

\begin{proposition} \label{CarAxKmbcciw} The following holds:
$$\begin{array}{lll}
\Kmbcciw & = & \{ \mathcal{B} \in \Kmbc \ : \  \ 
\mbox{$\mathcal{M}(\mathcal{B})$ validates (\axwci)} \}\\[2mm]
& = &  \{ \mathcal{B} \in \Ksw \ : \  \ 
\mbox{$\mathcal{M}(\mathcal{B})$ validates (\axtnd), (\axexp) and (\axwci)} \}.
\end{array}$$
\end{proposition}
\begin{proof} 
Let  $\mathcal{B}$ be a swap structure for \mbcciw, and let $\gamma=\cons\alpha \vee(\alpha \wedge \neg \alpha)$ be an instance of axiom (\axwci). Let
$v$ be a valuation over $\mathcal{B}$. Since $v(\alpha) \in \textsc{B}_\mathcal{A}^\axwci$ it follows that $\pi_3(v(\alpha)) \vee(\pi_1(v(\alpha)) \wedge \pi_2(v(\alpha)))=1$. By  the fact that $\mathcal{B} \in \Kmbc$ and by Definition~\ref{valNmat1} it follows that $\pi_1(v(\cons\alpha))=\pi_3(v(\alpha))$ and $\pi_1(v(\neg\alpha))=\pi_2(v(\alpha))$, whence $\pi_1(v(\cons\alpha \vee(\alpha \wedge \neg \alpha)))=\pi_3(v(\alpha)) \vee (\pi_1(v(\alpha)) \land \pi_2(v(\alpha))) = 1$. This means that $v(\gamma) \in D_\mathcal{B}$ for every instance $\gamma$ of axiom (\axwci).

Conversely, let $\mathcal{B} \in \Kmbc$ such that $\mathcal{M}(\mathcal{B})$ validates (\axwci), and let $p$ be a propositional variable. Let $z \in |\mathcal{B}|$, and  consider a valuation $v$ over  $\mathcal{B}$ such that $v(p)=z$. Reasoning as in the proof of Proposition~\ref{CarAxKmbc} it can be seen that $\pi_1(v(\cons p \vee(p \wedge \neg p)))=z_3 \vee (z_1 \wedge z_2)$. Since  $\mathcal{B}$ validates (\axwci), by  hypothesis, it follows that $z_3 \vee (z_1 \wedge z_2)=1$. That is,  $|\mathcal{B}| \subseteq \textsc{B}_\mathcal{A}^\axwci$, whence $\mathcal{B} \in \Kmbcciw$, by Definition~\ref{Swap-mbcciw}. 
\end{proof}

\begin{definition} 
The full subcategory in  \sw\ of swap structures for \mbcciw\ will be denoted by \swmbcciw. 
\end{definition}

By the very definitions,  \swmbcciw\ is  a full subcategory in \swmbc, and a full subcategory in  $\malg(\Sigma)$.
Hence,  the class of objects of \swmbcciw\  is $\Kmbcciw$, and the morphisms between two given swap structures for \mbcciw\ are just the homomorphisms between them as multialgebras over $\Sigma$.

\begin{definition} \label{maxBmbCciw}
Let  $\mathcal{A}$ be a  Boolean algebra. The {\em full swap structure for \mbcciw\ over $\mathcal{A}$}, denoted by $\mathcal{B}_\mathcal{A}^\mbcciw$, is the unique swap structure for \mbcciw\  over $\mathcal{A}$ with domain $\textsc{B}_\mathcal{A}^\axwci$ such that, for every $z$ and $w$ in $\textsc{B}_\mathcal{A}^\axwci$:

\begin{itemize}
 \item[(i)] $z\#w = \{u\in \textsc{B}_\mathcal{A}^\axwci \ : \ u_1=z_1\#w_1\}$, for each $\#\in \{\wedge,\vee, \to\}$;
\item[(ii)] $\neg (z) = \{u\in \textsc{B}_\mathcal{A}^\axwci \ : \ u_1=z_2\}$;
 \item[(iii)] $\circ (z) = \{u\in \textsc{B}_\mathcal{A}^\axwci \ : \ u_1=z_3\}$.\\[-2mm]
\end{itemize}
\end{definition}

The class $Mat(\Kmbcciw)$ of Nmatrices associated to swap structures for \mbcciw\ is defined analogously to the class $Mat(\Ksw)$ introduced in Definition~\ref{Nmatrix}. Let  $\models_{Mat(\Kmbcciw)}$ be the consequence relation associated to the class $Mat(\Kmbcciw)$ as in Definition~\ref{conseqclassNmat}. Then:

\begin{theorem} (\cite[Theorem 6.5.4]{CC16})
Let $\Gamma \cup \{\varphi\} \subseteq For(\Sigma)$  be a set of formulas. Then:
$\Gamma \vdash_\mbcciw \varphi$ \ iff \  $\Gamma\models_{Mat(\Kmbcciw)} \varphi$.
\end{theorem}

\begin{remark}
It is possible to give  a direct proof of the latter theorem, by extending the proof of Theorem~\ref{adeq-mbC} presented here. 
\end{remark}

Now, stronger extensions of \mbc\ will be analized:

\begin{definition}  \label{mbcci-eCPL} Consider the following extensions of \mbc:\\[2mm]
(1)  The logic \mbcci\ (\cite[Definition 3.1.7]{CC16}) is obtained from \mbc\  by adding the axiom schema 
$$\neg{\circ}\alpha \imp (\alpha \wedge \neg \alpha)       \hspace{2cm}     (\axci)$$
(2)  The logic \ci\ (\cite[Remark 3.5.18]{CC16}) is obtained from \mbcci\  by adding the axiom schema 
$$\neg\neg\alpha \imp \alpha       \hspace{2cm}     (\axcf)$$
(3)  The logic \eecpl\  is obtained from \mbc\  by adding the axiom schema 
$${\circ}\alpha       \hspace{4cm}     (\axcons)$$
\end{definition}

\

\begin{proposition} \label{car-extensions} 
(1)  The logic \mbcci\ properly extends \mbcciw, and \ci\ properly extends \mbcci.\\[1mm]
(2)  The logic \eecpl\ is a presentation of \cpl\ over $\Sigma$, in which the connective $\circ$ gives a top particle. Thus,  \eecpl\ properly extends \ci\ and it is semantically characterized by the usual 2-valued truth-tables for \cpl\ plus the operator $\circ(x)=1$ for every $x \in \{0,1\}$.
\end{proposition}
\begin{proof}
(1) For the first part, see \cite[Proposition 3.1.10]{CC16}. The second part can be proved analogously by considering bivaluations semantics for these logics, which is defined from the one for \mbc\ introduced in Definition~\ref{bivalold}. Details can be found in~\cite[Chapter~3]{CC16}).\\[1mm]
(2) Observe that, by (\axcons), (\axexp) and \MP,  the negation $\neg$ is explosive in \eecpl\ and so it coincides with the classical negation, by axiom (\axtnd). Since \cplp\ is included in \eecpl\ then this logic is nothing more than a presentation of \cpl\ by adding an unary connective  $\circ$ such that $\circ\alpha$ is a top particle for every  $\alpha$. The rest of the proof is obvious.  
\end{proof}

\begin{definition} \label{Kci}
(1) A {\em swap structure for \mbcci} is any $\mathcal{B} \in \Kmbcciw$ such that, for every $z \in |\mathcal{B}|$, $\cons(z) \defin \{({\sim}(z_1 \land z_2),z_1 \land z_2,1)\}$.  The class of swap structures for \mbcci\ will be denoted by \Kmbcci.\\[1mm]
(2) A {\em swap structure for \ci} is any $\mathcal{B} \in \Kmbcci$ such that, for every $z \in |\mathcal{B}|$, $\neg(z) \subseteq \{u\in |\mathcal{B}| \ : \ u_1=z_2 \ \mbox{ and } \  u_2  \leq z_1\}$.  The class of swap structures for \ci\ will be denoted by \Kci.\\
\end{definition}

\begin{definition} \label{maxBmbCci-Ci}
Let  $\mathcal{A}$ be a  Boolean algebra.\\[1mm]
(1) The {\em full swap structure for \mbcci\ over $\mathcal{A}$},  denoted by $\mathcal{B}_\mathcal{A}^\mbcci$, is the unique swap structure for \mbcci\ over $\mathcal{A}$ with domain $\textsc{B}_\mathcal{A}^\axwci$ such that, for every $z$ and $w$ in $\textsc{B}_\mathcal{A}^\axwci$:

\begin{itemize}
 \item[(i)] $z\#w = \{u\in \textsc{B}_\mathcal{A}^\axwci \ : \ u_1=z_1\#w_1\}$, for each $\#\in \{\wedge,\vee, \to\}$;
\item[(ii)] $\neg (z) = \{u\in \textsc{B}_\mathcal{A}^\axwci \ : \ u_1=z_2\}$;
 \item[(iii)] $\circ (z) = \{({\sim}(z_1 \land z_2),z_1 \land z_2,1)\}$.\\[-2mm]
\end{itemize}
(2) The {\em full swap structure for \ci\ over $\mathcal{A}$},  denoted by $\mathcal{B}_\mathcal{A}^\ci$, is the unique swap structure for \ci\ over $\mathcal{A}$ with domain $\textsc{B}_\mathcal{A}^\axwci$ such that the multioperations (other than $\neg$) are defined as in $\mathcal{B}_\mathcal{A}^\mbcci$ and, for every $z$ in $\textsc{B}_\mathcal{A}^\axwci$:

\begin{itemize}
\item[(ii)'\,] $\neg (z) = \{u\in \textsc{B}_\mathcal{A}^\axwci \ : \ u_1=z_2 \ \mbox{ and } \  u_2  \leq z_1\}$.\\[-2mm]
\end{itemize}
\end{definition}

The classes $Mat(\Kmbcci)$ and $Mat(\Kci)$ of Nmatrices  are defined analogously to the class $Mat(\Ksw)$ introduced in Definition~\ref{Nmatrix}. Thus:

\begin{theorem} (\cite[Theorems 6.5.11 and 6.5.23]{CC16})
Let $\Gamma \cup \{\varphi\} \subseteq For(\Sigma)$  be a set of formulas. Then:\\[1mm]
(1) $\Gamma \vdash_\mbcci \varphi$ \ iff \  $\Gamma\models_{Mat(\Kmbcci)} \varphi$.\\[1mm]
(2) $\Gamma \vdash_\ci \varphi$ \ iff \  $\Gamma\models_{Mat(\Kci)} \varphi$.
\end{theorem}

\begin{remark}
As in the case of \mbcciw, it is possible to give  a direct proof of the latter theorem, by extending the proof of Theorem~\ref{adeq-mbC}. 
\end{remark}

\begin{proposition} \label{CarAxKmbcci} The following holds:
$$\begin{array}{lll}
\Kmbcci & = & \{ \mathcal{B} \in \Kmbcciw \ : \  \ 
\mbox{$\mathcal{M}(\mathcal{B})$ validates (\axci)} \}\\[2mm]
& = & \{ \mathcal{B} \in \Kmbc \ : \  \ 
\mbox{$\mathcal{M}(\mathcal{B})$ validates (\axci)} \}\\[2mm]
& = &  \{ \mathcal{B} \in \Ksw \ : \  \ 
\mbox{$\mathcal{M}(\mathcal{B})$ validates (\axtnd), (\axexp) and (\axci)} \}.
\end{array}$$
and
$$\begin{array}{lll}
\Kci & = & \{ \mathcal{B} \in \Kmbcci \ : \  \ 
\mbox{$\mathcal{M}(\mathcal{B})$ validates (\axcf)} \}\\[2mm]
& = & \{ \mathcal{B} \in \Kmbcciw \ : \  \ 
\mbox{$\mathcal{M}(\mathcal{B})$ validates (\axci) and (\axcf)} \}\\[2mm]
& = & \{ \mathcal{B} \in \Kmbc \ : \  \ 
\mbox{$\mathcal{M}(\mathcal{B})$ validates (\axci) and (\axcf)} \}\\[2mm]
& = &  \{ \mathcal{B} \in \Ksw \ : \  \ 
\mbox{$\mathcal{M}(\mathcal{B})$ validates (\axtnd), (\axexp), (\axci) and (\axcf)} \}.
\end{array}$$
\end{proposition}
\begin{proof} Let us begin with \Kmbcci.
Let  $\mathcal{B}$ be a swap structure for \mbcci, and let $\gamma=\neg{\circ}\alpha \imp (\alpha \wedge \neg \alpha)$ be an instance of axiom (\axci). Let
$v$ be a valuation over $\mathcal{B}$, and let $z=v(\alpha)$. Given that $v(\cons\alpha)=({\sim}(z_1 \land z_2),z_1 \land z_2,1)$ and $v(\neg\cons\alpha) \in \{ w \in |\mathcal{B}| \ : \ w_1=\pi_2(v(\cons\alpha))\}$ then   $v(\neg\cons\alpha) \in \{ w \in |\mathcal{B}| \ : \ w_1=z_1 \land z_2\}$. On the other hand $v(\alpha \wedge \neg \alpha) \in \{ w \in |\mathcal{B}| \ : \ w_1=z_1 \land z_2\}$. Being so, $v(\gamma) \in \{ w \in |\mathcal{B}| \ : \ w_1=(z_1 \land z_2) \to (z_1 \land z_2)\} =  D_\mathcal{B}$ for every instance $\gamma$ of axiom (\axci).

Conversely, let $\mathcal{B} \in \Kmbcciw$ such that $\mathcal{M}(\mathcal{B})$ validates (\axci), and let $z \in |\mathcal{B}|$. Let $p$ be a propositional variable, and  consider a valuation $v$ over  $\mathcal{B}$ such that $v(p)=z$. Then $\pi_1(v(\neg\cons p \to(p \wedge \neg p)))=\pi_1(v(\neg\cons p)) \to \pi_1(v(p \wedge \neg p)) = \pi_2(v(\cons p)) \to (z_1 \wedge z_2) =1$, since  $\mathcal{B}$ validates (\axci). Hence,  $\pi_2(v(\cons p)) \leq z_1 \wedge z_2$. On the other hand, $\pi_1(v(\cons p)) = z_3={\sim}(z_1 \wedge z_2)$. Therefore $z_1 \wedge z_2={\sim}\pi_1(v(\cons p)) \leq \pi_2(v(\cons p))$ (by observing that, if $u \in \textsc{B}_\mathcal{A}^\axwci$ then $u_1 \vee u_2=1$ and so  ${\sim}u_1 \leq u_2$). That is, $\pi_2(v(\cons p)) = z_1 \wedge z_2$. This means that $\circ (z) = \{({\sim}(z_1 \land z_2),z_1 \land z_2,1)\}$, whence  $\mathcal{B} \in \Kmbcci$, by Definition~\ref{Kci}(1). 

Finally, let us analyze \Kci. Let  $\mathcal{B}$ be a swap structure for \ci, and let $\gamma=\neg\neg\alpha \imp \alpha$ be an instance of axiom (\axcf). Let
$v$ be a valuation over $\mathcal{B}$, and let $z=v(\alpha)$. Observe that $v(\neg \alpha) \in \{ u \in |\mathcal{B}| \ : \ u_1=z_2 \ \mbox{ and } \  u_2  \leq z_1\}$. From this, 
$v(\neg\neg \alpha) \in \neg v(\neg \alpha) \subseteq \{ w \in |\mathcal{B}| \ : \ w_1=\pi_2(v(\neg \alpha))\} \subseteq \{w\in |\mathcal{B}| \ : \ w_1 \leq z_1\}$.
Thus, $\pi_1(v(\gamma))= \pi_1(v(\neg\neg \alpha)) \to z_1 = 1$ and so $v(\gamma) \in D_\mathcal{B}$ for every instance $\gamma$ of axiom (\axcf).

Conversely, let $\mathcal{B} \in \Kmbcci$ such that $\mathcal{M}(\mathcal{B})$ validates (\axcf). Let $z \in |\mathcal{B}|$ and $u \in \neg(z)$. Let $p$ be a propositional variable, and  consider a valuation $v$ over  $\mathcal{B}$ such that $v(p)=z$ and $v(\neg p)=u$. Then $v(\neg\neg p) \in \neg v(\neg p) = \neg u$, whence $\pi_1(v(\neg\neg p))=u_2$. From this $\pi_1(v(\neg\neg p \to p)) =  \pi_1(v(\neg\neg p)) \to z_1 = u_2 \to z_1=1$, provided that $\mathcal{M}(\mathcal{B})$ validates (\axcf). Therefore $u_2 \leq z_1$. This means that $\neg (z) \subseteq \{u\in |\mathcal{B}| \ : \ u_1=z_2 \ \mbox{ and } \  u_2  \leq z_1\}$, whence  $\mathcal{B} \in \Kci$, by Definition~\ref{Kci}(2).
\end{proof}

Finally \eecpl, classical propositional logic defined over $\Sigma$, will be characterized by means of swap structures.

\begin{definition} 
Let $\mathcal{A}$ be a Boolean algebra with domain $A$. The universe of swap structures for \eecpl\ over $\mathcal{A}$ is the set 
$$\textsc{B}_\mathcal{A}^{\eecpl}= \{z\in \textsc{B}_\mathcal{A}^\axwci \ : \ z_2 = {\sim} z_1\}  = \{(a,{\sim}a,1) \ : \ a \in A\} \simeq A.$$
\end{definition}

\begin{definition} 
A {\em swap structure for \eecpl} is any $\mathcal{B} \in \Kci$ such that $|\mathcal{B}|\subseteq \textsc{B}_\mathcal{A}^{\eecpl}$.   The class of swap structures for \eecpl\ will be denoted by \Keecpl.
\end{definition}

\begin{proposition} \label{CarAxKeecpl} The following holds:
$$\begin{array}{lll}
\Keecpl & = & \{ \mathcal{B} \in \Kci \ : \ \ 
\mbox{$\mathcal{M}(\mathcal{B})$ validates (\textbf{cons})} \}\\[2mm]
& = & \{  \mathcal{B} \in \Kmbc \ : \ \ 
\mbox{$\mathcal{M}(\mathcal{B})$ validates (\textbf{cons})} \}\\[2mm]
& = &  \{ \mathcal{B} \in \Ksw \ : \ \ 
\mbox{$\mathcal{M}(\mathcal{B})$ validates (\axtnd), (\axexp) and (\textbf{cons})} \}.
\end{array}$$
\end{proposition}
\begin{proof}
Let  $\mathcal{B} \in \Keecpl$, and let $\gamma=\cons\alpha$ be an instance of axiom (\textbf{cons}). Let
$v$ be a valuation over $\mathcal{B}$, and let $z=v(\alpha)$. Given that $z \in \textsc{B}_\mathcal{A}^{\eecpl}$ then  $z_1 \land z_2 =0$. Since $v(\cons \alpha) = ({\sim}(z_1 \land z_2),z_1 \land z_2,1)$ then  $v(\gamma) \in D_\mathcal{B}$ for every instance $\gamma$ of axiom (\textbf{cons}).

Now, let $\mathcal{B} \in \Kmbc$ such that $\mathcal{M}(\mathcal{B})$ validates (\textbf{cons}), and let $p$ and $q$ be two different propositional variables. Let $z \in |\mathcal{B}|$, and  consider a valuation $v$ over  $\mathcal{B}$ such that $v(p)=z$ and $\pi_1(v(q))=0$ (this is always possible since, by Definition~\ref{Swap-str}, $0 \in \pi_1[|\mathcal{B}|]$). As in the proof of Proposition~\ref{CarAxKmbc} it follows that $\pi_3(v(p)) \to (\pi_1(v(p)) \to (\pi_2(v(p)) \to 0))=1$. But $z_3=\pi_3(v(p)) = \pi_1(v(\cons p))=1$, since  $\mathcal{M}(\mathcal{B})$ validates (\textbf{cons}). Therefore $\pi_1(v(p)) \to (\pi_2(v(p)) \to 0)=1$ and so $\pi_1(v(p)) \land \pi_2(v(p))=0$. That is, $z_1 \land z_2=0$, whence  $z_2 ={\sim}z_1$. This means that $z \in \textsc{B}_\mathcal{A}^{\eecpl}$ and so $|\mathcal{B}| \subseteq \textsc{B}_\mathcal{A}^{\eecpl}$. From this is straightforward to see that $\mathcal{B} \in \Kci$, therefore  $\mathcal{B} \in \Keecpl$.
\end{proof}

\

The full subcategory in  \sw\ of swap structures for \mbcci\ and for \eecpl\ will be denoted by \swmbcci\ and \sweecpl, respectively. 
By the very definitions, they are full subcategories in \swmbc, and full subcategories in  $\malg(\Sigma)$.

\begin{remark} \label{obsalge}
(1) If $\mathcal{B}  \in \Keecpl$ then $\mathcal{B}$ can be seen as a Boolean algebra isomorphic to the Boolean algebra $\pi_1[|\mathcal{B}|]$. Indeed, $(a,{\sim}a,1) \mapsto a$ is a bijection. On the other hand, the operations in $\mathcal{B}$ are defined as follows, for every $(a,{\sim}a,1)$ and $(b,{\sim}b,1)$ in $|\mathcal{B}|$:
\begin{itemize}
 \item[(i)] $(a,{\sim}a,1)\# (b,{\sim}b,1) = \{(a\#b,{\sim}(a\#b),1)\}$, for each $\#\in \{\wedge,\vee, \to\}$;
\item[(ii)] $\neg (a,{\sim}a,1) =\{({\sim}a,a,1)\}$;
 \item[(iii)] $\circ (a,{\sim}a,1) = \{(1,0,1)\}$.
\end{itemize}
(2) Observe that $$\Keecpl \subset \Kci \subset \Kmbcci \subset \Kmbcciw \subset \Kmbc \subset \Ksw$$ 
while 
$$\eecpl \supset \ci \supset \mbcci \supset  \mbcciw \supset \mbc \supset \ecpl.$$
\end{remark}

As analyzed in \cite[Chapter 6]{CC16}, the logic \mbcciw\ can be characterized by a single 3-valued Nmatrix, by considering the full swap structure over the two-valued Boolean algebra $\mathbb{A}_2$. Indeed the Nmatrix $\mathcal{M}_3^\mbcciw$ induced by the full swap structure $\mathcal{B}_{\mathbb{A}_2}^\mbcciw$ (recall  Definition~\ref{maxBmbCciw}) was originally considered by A. Avron in  \cite{avr:05},  obtaining so a semantical characterization of \mbcciw. The domain of the multialgebra $\mathcal{B}_{\mathbb{A}_2}^\mbcciw$ is the set  $\textsc{B}_{\mathbb{A}_2}^\axwci = \big\{T, \, t, \,  F\big\}$ such that $T=(1,0,1)$, $t=(1,1,0)$ and $F=  (0,1,1)$, where  $\textrm{D}_3=\{T, \, t\}$ is the set of designated values. 
The multioperations are defined as follows:

\ 

\begin{center}
\begin{tabular}{|c|c|c|c|}
\hline
 $\land$ & $T$   & $t$ & $F$ \\
 \hline \hline
    $T$    & $\{t,T\}$   & $\{t,T\}$   & $\{F\}$   \\ \hline
     $t$    & $\{t,T\}$   & $\{t,T\}$   & $\{F\}$  \\ \hline
     $F$    & $\{F\}$  & $\{F\}$ & $\{F\}$  \\ \hline
\end{tabular}
\hspace{1cm}
\begin{tabular}{|c|c|c|c|}
\hline
 $\lor$ & $T$   & $t$  & $F$  \\
 \hline \hline
    $T$    & $\{t,T\}$   & $\{t,T\}$  & $\{t,T\}$    \\ \hline
     $t$    & $\{t,T\}$   & $\{t,T\}$  & $\{t,T\}$   \\ \hline
     $F$    & $\{t,T\}$   & $\{t,T\}$  & $\{F\}$   \\ \hline
\end{tabular}
\end{center}

\

\begin{center}
\begin{tabular}{|c|c|c|c|}
\hline
 $\imp$ & $T$   & $t$ & $F$ \\
 \hline \hline
    $T$    & $\{t,T\}$   & $\{t,T\}$  & $\{F\}$   \\ \hline
     $t$    & $\{t,T\}$   & $\{t,T\}$  & $\{F\}$  \\ \hline
     $F$    & $\{t,T\}$   & $\{t,T\}$  & $\{t,T\}$   \\ \hline
\end{tabular}
\hspace{1cm}
\begin{tabular}{|c||c|} \hline
$\quad$ & $\neg$ \\
 \hline \hline
    $T$   & $\{F\}$    \\ \hline
     $t$   & $\{t,T\}$     \\ \hline
     $F$   & $\{t,T\}$     \\ \hline
\end{tabular}
\hspace{1cm}
\begin{tabular}{|c||c|}
\hline
 $\quad$ & $\cons$ \\
 \hline \hline
    $T$   & $\{t,T\}$     \\ \hline
     $t$   & $\{F\}$    \\ \hline
     $F$   & $\{t,T\}$     \\ \hline
\end{tabular}
\end{center}

\ \\

It is clear that $\mathcal{B}_{\mathbb{A}_2}^\mbcciw$ is a submultialgebra of $\mathcal{B}_{\mathbb{A}_2}^\mbc$. Moreover, by an analysis similar to the one presented above, it is possible to prove the following:

\begin{theorem} [Representation Theorem for \Kmbcciw]  Let $\mathcal{B}$ be a swap structure  for \mbcciw. Then, there exists a set $I$ and a monomorphism  of multialgebras $\hat h:\mathcal{B} \to \prod_{i \in I}\mathcal{B}_{\mathbb{A}_2}^\mbcciw$.
\end{theorem}

Concerning \mbcci\ and \ci, similar results can be obtained. Indeed,  A. Avron has proven in  \cite{avr:05} that \mbcci\ can be characterized by a single 3-valued Nmatrix. In \cite[Chapter 6]{CC16} it was proved that Avron's Nmatrix is exactly the one obtained from the 3-valued full swap structure
$\mathcal{B}_{\mathbb{A}_2}^\mbcci$ over $\mathbb{A}_2$ (see Definition~\ref{maxBmbCci-Ci}(1)). The full swap structure $\mathcal{B}_{\mathbb{A}_2}^\mbcci$ coincides with $\mathcal{B}_{\mathbb{A}_2}^\mbcciw$ with exception of $\circ$. Indeed, in $\mathcal{B}_{\mathbb{A}_2}^\mbcci$  the multioperator  $\circ$ is now single-valued, and it is defined as follows:

\ 

\begin{center}
\begin{tabular}{|c||c|}
\hline
 $\quad$ & $\cons$ \\
 \hline \hline
    $T$   & $\{T\}$     \\ \hline
     $t$   & $\{F\}$    \\ \hline
     $F$   & $\{T\}$     \\ \hline
\end{tabular}
\end{center}

\ \\

Clearly, $\mathcal{B}_{\mathbb{A}_2}^\mbcci$ is a submultialgebra of $\mathcal{B}_{\mathbb{A}_2}^\mbcciw$ and so of $\mathcal{B}_{\mathbb{A}_2}^\mbc$. Moreover:

\begin{theorem} [Representation Theorem for \Kmbcci]  Let $\mathcal{B}$ be a swap structure  for \mbcci. Then, there exists a set $I$ and a monomorphism  of multialgebras $\hat h:\mathcal{B} \to \prod_{i \in I}\mathcal{B}_{\mathbb{A}_2}^\mbcci$.
\end{theorem}

Consider now \ci. In~\cite{avr:05} A. Avron has obtained a semantical characterization of \ci\ in terms of a single 3-valued Nmatrix $\mathcal{M}_\ci$. In \cite[Chapter 6]{CC16} it was shown that the underlying multialgebra of $\mathcal{M}_\ci$ is $\mathcal{B}_{\mathbb{A}_2}^\ci$, the full swap structure for \ci\  over $\mathbb{A}_2$ (see Definition~\ref{maxBmbCci-Ci}(2)). This multialgebra coincides with 
$\mathcal{B}_{\mathbb{A}_2}^\mbcci$ with exception of  the multioperator $\neg$, which  is now defined as follows:

\ 

\begin{center}
\begin{tabular}{|c||c|} \hline
$\quad$ & $\neg$ \\
 \hline \hline
    $T$   & $\{F\}$    \\ \hline
     $t$   & $\{t,T\}$     \\ \hline
     $F$   & $\{T\}$     \\ \hline
\end{tabular}
\end{center}

\ \\

It is clear that $\mathcal{B}_{\mathbb{A}_2}^\ci$ is a submultialgebra of $\mathcal{B}_{\mathbb{A}_2}^\mbcci$ and so of $\mathcal{B}_{\mathbb{A}_2}^\mbcciw$ and $\mathcal{B}_{\mathbb{A}_2}^\mbc$. Moreover, the following representation result holds:

\begin{theorem} [Representation Theorem for \Kci]  Let $\mathcal{B}$ be a swap structure  for \ci. Then, there exists a set $I$ and a monomorphism  of multialgebras $\hat h:\mathcal{B} \to \prod_{i \in I}\mathcal{B}_{\mathbb{A}_2}^\ci$.
\end{theorem}

Finally, the case of \eecpl\ is quite simple. By Remark~\ref{obsalge}(1), there is only one swap structure for \eecpl\ with domain $\textsc{B}_\mathcal{A}^{\eecpl}$, which is precisely the full swap structure denoted by $\mathcal{B}_\mathcal{A}^{\eecpl}$.
In particular, the swap structure $\mathcal{B}_{\mathbb{A}_2}^{\eecpl}$ has domain $\{T,F\}$ where $T=(1,0,1)$ and $F=(0,1,1)$. The multioperations are single-valued, producing a Boolean algebra isomorphic to $\mathbb{A}_2$, by Remark~\ref{obsalge}(1). Using the notation introduced in Definition~\ref{subal} it is clear that
$$\mathcal{B}_{\mathbb{A}_2}^{\eecpl} \subseteq \mathcal{B}_{\mathbb{A}_2}^\ci \subseteq \mathcal{B}_{\mathbb{A}_2}^\mbcci \subseteq \mathcal{B}_{\mathbb{A}_2}^\mbcciw \subseteq\mathcal{B}_{\mathbb{A}_2}^\mbc  \subseteq\mathcal{B}_{\mathbb{A}_2}^{\ecpl}.$$
Additionally:

\begin{theorem} [Representation Theorem for \Keecpl]  Let $\mathcal{B}$ be a swap structure  for \eecpl. Then, there exists a set $I$ and a monomorphism  of algebras $\hat h:\mathcal{B} \to \prod_{i \in I}\mathcal{B}_{\mathbb{A}_2}^{\eecpl}$.
\end{theorem}

The last theorem is just the original G. Birkhoff's theorem for Boolean algebras~\cite{Birk:35}, under a different presentation.

\begin{remark} \label{pairs}
Recall from Definition~\ref{univ-sw-mbcciw} that the universe of swap structures  for \mbcciw\ over $\mathcal{A}$ is  $\textsc{B}_\mathcal{A}^\axwci=\{z \in A^3 \ :  \ z_1 \vee z_2=1 \ \mbox{ and } \  z_3 =  {\sim} (z_1 \wedge z_2)\}$. Thus, the third coordinate of the snapshots is defined in terms of the other two, being so redundant. This means that, in swap structures for \mbcciw\ and its extensions,  the snapshots could be considered as being pairs instead of triples. This feature is obvious in the case of  \eecpl, in which any snapshot $(a,{\sim}a,1)$ could be represented as $(a,{\sim}a)$ (or simply by $a$ itself). As it will be discussed in the next section, this fact evidences the close relationship between swap structures and the so-called {\em twist structures}.
\end{remark}

\

\section{Twist structures as special cases of swap structures} \label{sectwist-swap}

The swap structures semantics for some \lfis\ presented in the previous sections was based on multialgebras since the given logics are not algebraizable in the classical sense. Being so, multialgebras arise as a natural alternative to algebras. In sections~\ref{secLFI1} and~\ref{sectCiore} the same techniques will be applied to algebraizable logics which are characterized by a single 3-valued logical matrix. It will be seen that the algebras associated to these logics will be recovered as special cases of swap structures, obtaining so an interesting relationship with the twist-structures semantics. This connection suggest that swap structures can be seen as non-deterministic twist structures, as it will be argued in Section~\ref{swap=twist} below.

\subsection{Swap structures for \dacdot: restoring determinism} \label{secLFI1}

The logic \dacdot\ was introduced in 1970 by I.~M.~L.  D'Ottaviano and N.~C.~A. da Costa  as a 3-valued modal logic (see~\cite{dot:dac:70}). Afterwards, this logic has been re-introduced independently by several authors, presented in different signatures. For instance, it was re-discovered in  2000 by W. Carnielli, J. Marcos and S. de Amo  as  a 3-valued \lfi\ called \lfium, apt to deal with inconsistent databases (see~\cite{car:mar:dea:00}). More recently,  M. Coniglio and L. Silvestrini propose in~\cite{Con:Sil} a generalization of the  notion of quasi-truth (see~\cite{mik:costa:chu:86}) based on a 3-valued paraconsistent logic called \mpt\ with was proved to be equivalent, up to laguage, with \dacdot\ (and so to \lfium). More historical remarks about this logic can be found in~\cite[Chapter~4]{CC16}.

A new axiomatization of this logic, presented as an \lfi\ over signature $\Sigma$,  was proposed in~\cite{CC16} under the name of $\lfium_\cons$. For the sake of convenience, this will be the presentation of this logic to be adopted here. From now on we will write $\alpha \leftrightarrow \beta$ as an abbreviation of the formula $(\alpha \to \beta) \wedge (\beta \to \alpha)$. 

\begin{definition} \label{LFI1-lfi} (\cite[Definition~4.4.41]{CC16}) Let $\lfium_\cons$ be the logic over $\Sigma$ obtained from \ci\   (see Definition~\ref{mbcci-eCPL}(2))  by adding the following axiom schemas:
$$\begin{array}{ll}
\alpha \imp \neg\neg\alpha &    \hspace{2cm}        (\axce)\\[2mm]
\neg(\alpha \lor \beta) \leftrightarrow (\neg \alpha \land \neg\beta) & \hspace{2cm} (\axnegou)\\[2mm]
\neg(\alpha \land \beta) \leftrightarrow (\neg \alpha \lor \neg\beta) & \hspace{2cm} (\axnege)\\[2mm]
\neg(\alpha \imp \beta) \leftrightarrow (\alpha \land \neg\beta)   & \hspace{2cm} (\axnegimp)\\[2mm]
\end{array}$$
\end{definition}

As proven in~\cite[Theorem~4.4.45]{CC16}, the logic $\lfium_\cons$ is semantically characterized by a 3-valued logical matrix with domain  $\textsc{B}_{\mathbb{A}_2}^\axwci = \big\{T, \, t, \,  F\big\}$ such that  $\textrm{D}_3=\{T, \, t\}$ is the set of designated values. 
The operations are defined as follows:

\ 

\begin{center}
\begin{tabular}{|c|c|c|c|}
\hline
 $\land$ & $T$   & $t$ & $F$ \\
 \hline \hline
    $T$    & $T$   & $t$   & $F$   \\ \hline
     $t$    & $t$   & $t$   & $F$  \\ \hline
     $F$    & $F$  & $F$ & $F$  \\ \hline
\end{tabular}
\hspace{1cm}
\begin{tabular}{|c|c|c|c|}
\hline
 $\lor$ & $T$   & $t$  & $F$  \\
 \hline \hline
    $T$    & $T$   & $T$  & $T$    \\ \hline
     $t$    & $T$   & $t$  & $t$   \\ \hline
     $F$    & $T$   & $t$  & $F$   \\ \hline
\end{tabular}
\end{center}

\

\begin{center}
\begin{tabular}{|c|c|c|c|}
\hline
 $\imp$ & $T$   & $t$ & $F$ \\
 \hline \hline
    $T$    & $T$   & $t$  & $F$   \\ \hline
     $t$    & $T$   & $t$  & $F$  \\ \hline
     $F$    & $T$   & $T$  & $T$   \\ \hline
\end{tabular}
\hspace{1cm}
\begin{tabular}{|c||c|} \hline
$\quad$ & $\neg$ \\
 \hline \hline
    $T$   & $F$    \\ \hline
     $t$   & $t$     \\ \hline
     $F$   & $T$     \\ \hline
\end{tabular}
\hspace{1cm}
\begin{tabular}{|c||c|}
\hline
 $\quad$ & $\cons$ \\
 \hline \hline
    $T$   & $T$     \\ \hline
     $t$   & $F$    \\ \hline
     $F$   & $T$     \\ \hline
\end{tabular}
\end{center}

\ \\

This logical matrix corresponds to the usual presentation of \lfium\ as a 3-valued logic over signature $\Sigma$, and it is equivalent to  \dacdot\ up to language, as mentioned above.

Taking into account Remark~\ref{pairs}, in order to simplify the presentation of swap structures for $\lfium_\cons$ the snapshots will taken as pairs instead of triples.
That is, along the rest of this paper the  universe of swap structures for \mbcciw\ and its extensions will be the set $\textsc{B}_\mathcal{A}^\axwci=\{z \in A^2 \ : \ z_1 \vee z_2=1  \}$. In particular, the universe of the swap structures over the two-element Boolean algebra $\mathbb{A}_2$ will be the set  $\textsc{B}_{\mathbb{A}_2}^\axwci = \big\{T, \, t, \,  F\big\}$ such that $T=(1,0)$, $t=(1,1)$ and $F=  (0,1)$. The elements of $\textsc{B}_{\mathbb{A}_2}^\axwci$ can be identified with the elements of the logical matrix of \lfium\ described above (which justifies the use of the same notation for both structures).

By using the axioms of $\lfium_\cons$  we arrive to the following definition, which will be rigorously justified by Proposition~\ref{CarAxKlfium} below:

\begin{definition}  \label{defKlfi1}
A {\em swap structure for $\lfium_\cons$} is any $\mathcal{B} \in \Kci$ such that the multioperations are single-valued and defined as follows, for every $(z_1,z_2),(w_1,w_2) \in |\mathcal{B}|$:

\begin{itemize}
 \item[(i)] $(z_1,z_2)\wedge (w_1,w_2) = \{(z_1 \wedge w_1,z_2 \vee w_2) \}$;
 \item[(ii)]  $(z_1,z_2)\vee (w_1,w_2) = \{(z_1 \vee w_1,z_2 \wedge w_2) \}$;
  \item[(iii)]  $(z_1,z_2)\imp (w_1,w_2) = \{(z_1 \imp w_1,z_1 \wedge w_2) \}$;
\item[(iv)] $\neg (z_1,z_2) = \{(z_2,z_1)\}$;
 \item[(iii)] $\circ (z_1,z_2) = \{({\sim}(z_1 \wedge z_2),z_1 \wedge z_2)\}$.
\end{itemize}
The class of swap structures for $\lfium_\cons$ will be denoted by \Klfium.
\end{definition}

\begin{remark} \label{obs-twist-J3}
It is interesting to notice the similarity between the swap structures for $\lfium_\cons$  and the twist structures for paraconsistent Nelson's logic \nel\ considered by S. Odintsov in~\cite{odin:03}. There are two differences between both structures: on the one hand, the latter are defined over implicative lattices, while the former are defined over Boolean algebras (which are implicative lattices with a bottom element satisfying additionally that $a \vee (a \to b)=1$ for every $a,b$, recall Proposition~\ref{BA-HA}). On the other hand, the former are an expansion of the latter by adding the unary operator $\circ$.
This should not be surprising since this fact already appears at the syntactical presentation of the logics as Hilbert calculi: $\lfium_\cons$ is obtained from \nel\ by adding axioms (\axouimp) and (\axtnd) plus the consistency operator $\cons$ governed by axioms (\axexp) and (\axci).
As a matter of fact, it is worth noting that $\lfium_\cons$ (and so \dacdot) can be presented over the signature $\Sigma_0=\{\land, \lor, \to, \neg, \bot\}$, where $\bot$ is a constant for denoting the bottom element. Thus,  in this signature $\lfium_\cons$ corresponds to an axiomatic extension of $\nel^\bot$ (the expansion of \nel\ by adding a bottom $\bot$, see~\cite[Section~8.6]{odin:08}) in which the consistency operator is defined  as $\cons\alpha\defin {\sim}(\alpha \land \neg \alpha)$, where ${\sim}\alpha\defin \alpha \to \bot$. The swap/twist structures for this presentation of $\lfium_\cons$  are defined as in Definition~\ref{defKlfi1}, by taking $\bot \defin (0,1)$ (hence  ${\sim}(z_1,z_2) = ({\sim}z_1,z_1)$). The close relationship between swap structures and twist structures will be analyzed with more detail in sections~\ref{Kalman-twist} and~\ref{swap=twist}.
\end{remark}


\begin{definition} \label{maxLFI1}
Given a Boolean algebra $\mathcal{A}$, the {\em full swap structure  for $\lfium_\cons$ over $\mathcal{A}$}, denoted by $\mathcal{B}_\mathcal{A}^{\lfium_\cons}$, is the unique swap structure  for $\lfium_\cons$ defined over $\mathcal{A}$ with domain $\textsc{B}_\mathcal{A}^\axwci$.
\end{definition}

\begin{proposition} \label{CarAxKlfium} Let $\mathbf{Ax}$ be the set of axioms added to \ci\ in order to obtain $\lfium_\cons$ (recall Definition~\ref{LFI1-lfi}). Then:
$$\begin{array}{lll}
\Klfium & = & \{ \mathcal{B} \in \Kci \ : \ \ 
\mbox{$\mathcal{M}(\mathcal{B})$ validates all the axioms in $\mathbf{Ax}$} \}.
\end{array}$$
\end{proposition}
\begin{proof}
{\bf Part 1:} If  $\mathcal{B} \in \Klfium$ then $\mathcal{B} \in \Kci$ such that $\mathcal{M}(\mathcal{B})$ validates all the axioms in $\mathbf{Ax}$.\\
Let $\mathcal{B} \in \Klfium$, and let $v$ be a valuation over $\mathcal{B}$.
Let  $\gamma=\alpha \imp \neg\neg\alpha$ be an instance of axiom (\axce), and let $z=v(\alpha)$. Then $v(\neg \alpha)=(z_2,z_1)$ and so $v(\neg\neg \alpha)=z=v(\alpha)$. From this, $\pi_1(v(\gamma))= z_1 \to \pi_1(v(\neg\neg \alpha)) = z_1 \to z_1 = 1$ and so $v(\gamma) \in D_\mathcal{B}$ for every instance $\gamma$ of axiom (\axce).

Now, let $\gamma'= \neg(\alpha \lor \beta) \leftrightarrow (\neg \alpha \land \neg\beta)$ be an instance of axiom (\axnegou). Let $z=v(\alpha)$ and $w=v(\beta)$. Then $v(\alpha \lor \beta)=(z_1 \vee w_1,z_2 \land w_2)$ and so $v(\neg(\alpha \lor \beta))=(z_2 \land w_2,z_1 \vee w_1)$. On the other hand $v(\neg\alpha)=(z_2,z_1)$ and $v(\neg\beta)=(w_2,w_1)$, and so $v(\neg \alpha \land \neg\beta) = (z_2 \land w_2,z_1 \vee w_1) = v(\neg(\alpha \lor \beta))$. Thus, $\pi_1(v(\gamma'))=1$ for every instance $\gamma'$ of axiom (\axnegou). Analogously, it can be proven that $\mathcal{B}$ validates all the other axioms in $\mathbf{Ax}$.\\[2mm]
{\bf Part 2:} If  $\mathcal{B} \in \Kci$ such that $\mathcal{M}(\mathcal{B})$ validates all the axioms in $\mathbf{Ax}$ then $\mathcal{B} \in \Klfium$.\\
Fix $\mathcal{B} \in \Kci$ such that $\mathcal{M}(\mathcal{B})$ validates all the axioms in $\mathbf{Ax}$. Let $z \in |\mathcal{B}|$ and $u \in \neg(z)$. Then $u_1=z_2$ and $u_2 \leq z_1$, by Definition~\ref{Kci}(2). On the other hand, the  validation of axiom~(\axce) forces to have $z_1 \leq u_2$ and so $u_2=z_1$ That is, $\neg (z_1,z_2) = \{(z_2,z_1)\}$. 

With respect to the disjunction multioperator,  let $z,w,u \ \in |\mathcal{B}|$ such that $u \in z \vee w$.  By Definition~\ref{Kci}(2) it follows that $u_1=z_1 \vee w_1$. Consider two different propositional variables $p,q$ and a valuation $v$ over $\mathcal{B}$ such that $v(p)=z$, $v(q)=w$ and $v(p \vee q)=u$. Then $\pi_1(v(\neg p \land \neg q))= \pi_1(v(\neg p)) \land \pi_1(v(\neg q)) = \pi_2(v(p)) \land \pi_2(v(q)) = z_2 \land w_2$. On the other hand, $\pi_1(v(\neg(p \vee q))) = \pi_2(v(p \vee q)) = u_2$. By axiom $(\axnegou)$, $\pi_1(v(\neg p \land \neg q))=\pi_1(v(\neg(p \vee q)))$ and so $u_2 = z_2 \land w_2$. This means that $z \vee w = \{(z_1 \vee w_1,z_2 \wedge w_2)\}$ for every $z,w$.

The other multioperations are treated in the same way. The details are left to the reader.
\end{proof}

\

The class $Mat(\Klfium)$ of Nmatrices is defined analogously to the class $Mat(\Ksw)$ introduced in Definition~\ref{Nmatrix}.
The adequacy of $\lfium_\cons$ w.r.t. swap structures can be proven by extending the proof of Theorem~\ref{adeq-mbC} for \mbc.

\begin{theorem}  [Adequacy of $\lfium_\cons$ w.r.t. swap structures] \label{adeqLFIum}
Let $\Gamma \cup \{\varphi\} \subseteq For(\Sigma)$  be a set of formulas. Then:
$\Gamma \vdash_{\lfium_\cons} \varphi$ \ iff \  $\Gamma\models_{Mat(\Klfium)} \varphi$.
\end{theorem}
\begin{proof}
The proof is similar to that for Theorem~\ref{adeq-mbC}.\\
`Only if' part (Soundness): It is a consequence of Proposition~\ref{CarAxKeecpl} and the fact that trueness is preserved by~(\MP).\\ [2mm]
`If' part (Completeness): Suppose that $\Gamma \nvdash_{\lfium_\cons} \varphi$. Define in $For(\Sigma)$ the following relation: $\alpha \equiv_\Gamma \beta$ iff $\Gamma \vdash_{\lfium_\cons} \alpha \imp \beta$ and $\Gamma \vdash_{\lfium_\cons} \beta \imp \alpha$. As in the proof of  Theorem~\ref{adeq-mbC} it follows that $\equiv_\Gamma$ is an equivalence relation such that $A_\Gamma \defin For(\Sigma)/_{\equiv_\Gamma}$ is the domain of a Boolean algebra $\mathcal{A}_\Gamma$ in which $[\alpha]_{\Gamma} \,\#\, [\beta]_{\Gamma} \defin [\alpha \# \beta]_{\Gamma}$, for $\# \in \{\land,\lor,\imp\}$, $0_\Gamma \defin [p_1 \wedge \neg p_1 \wedge \cons p_1]_\Gamma$ and $1_\Gamma \defin [p_1 \to p_1]_\Gamma$. Let  $\mathcal{B}_{\mathcal{A}_\Gamma}^{\lfium_\cons}$ be the corresponding full swap structure for $\lfium_\cons$ (recall Definition~\ref{maxLFI1}), 
and let $\mathcal{M}_\Gamma^{\lfium_\cons} \defin\mathcal{M}(\mathcal{B}_{\mathcal{A}_\Gamma}^{\lfium_\cons})$. The mapping $v_\Gamma:For(\Sigma) \to  \textsc{B}_{\mathcal{A}_\Gamma}^{\lfium_\cons}$ given by $v_\Gamma(\alpha) =([\alpha]_{\Gamma},[\neg \alpha]_{\Gamma})$  is a valuation over the Nmatrix $\mathcal{M}_\Gamma^{\lfium_\cons}$ such that $v_\Gamma(\alpha) \in D_{\mathcal{B}_{\mathcal{A}_\Gamma}^{\lfium_\cons}}$ iff $\Gamma \vdash_{\lfium_\cons} \alpha$, for every $\alpha$. From this, $v_\Gamma[\Gamma] \subseteq D_{\mathcal{B}_{\mathcal{A}_\Gamma}^{\lfium_\cons}}$ but  $v_\Gamma(\varphi) \not\in D_{\mathcal{B}_{\mathcal{A}_\Gamma}^{\lfium_\cons}}$. Therefore $\Gamma\not\models_{Mat(\Klfium)} \varphi$, by Definition~\ref{conseqclassNmat}.
\end{proof}

\

Let $\mathcal{B}_{\mathbb{A}_2}^{\lfium_\cons}$  be the full swap structure for $\lfium_\cons$ over $\mathbb{A}_2$. Clearly it is equivalent to the 3-valued logical matrix for \lfium\ presented above, in which any truth-value $z$ is replaced by the singleton $\{z\}$ on each entry of the tables (that is, by considering each operator as a single-valued multioperator). 
By using a technique similar to the one employed by \mbc\ and the other \lfis\ analyzed in the previous sections, it will be proven the adequacy of $\lfium_\cons$ w.r.t. the 3-valued Nmatrix  $\mathcal{B}_{\mathbb{A}_2}^{\lfium_\cons}$, see  Theorem~\ref{adeqLFIumA2} below. Clearly, this result corresponds to the adequacy of $\lfium_\cons$ w.r.t. the 3-valued standard logical matrix for \lfium/\dacdot\ (see~\cite[Theorem~4.4.45]{CC16}).

\begin{definition} [\cite{CC16}] \label{bivalLFI1}
A bivaluation $\mu:For(\Sigma)\to \big\{0,1\big\}$ for \mbc\ (recall Definition~\ref{bivalold}) is a {\em bivaluation for $\lfium_\cons$}
if it  satisfies in addition the   following clauses:\\[2mm]
{\bf (\valCi)} \ $\mu(\neg\cons \alpha) = 1$ \ implies \ $\mu(\alpha)=\mu(\neg \alpha)=1$ \\[2mm]
{\bf (\valcef)} \ $\mu(\lnot\lnot \alpha)=1$ \ iff  \ $\mu(\alpha)=1$ \\[2mm]
{\bf (\valdm$_\land$)} \ $\mu(\neg(\alpha\land \beta))=1$ \ iff  \ $\mu(\neg\alpha)=1$ \ or \ $\mu(\neg\beta)=1$.\\[2mm]
{\bf (\valdm$_\lor$)} \ $\mu(\neg(\alpha \lor\beta))=1$ \ iff  \ $\mu(\neg \alpha)=\mu(\neg \beta)=1$.\\[2mm]
{\bf (\valcip$_\imp$)} \ $\mu(\neg(\alpha\imp \beta))=1$ \ iff  \ $\mu(\alpha)=\mu(\neg\beta)=1$. \\[2mm]
The consequence relation of $\lfium_\cons$ w.r.t. bivaluations is defined as follows: for every set of formulas $\Gamma \cup \{\varphi\} \subseteq For(\Sigma)$, $\Gamma \models_{\lfium_\cons}^2 \varphi$ \ iff \ $\mu(\varphi)=1$ for every bivaluation for $\lfium_\cons$ such that $\mu[\Gamma] \subseteq \{1\}$.
\end{definition}

\begin{theorem} [\cite{CCM}]  \label{comp-bival-LFI1}
For every set of formulas $\Gamma \cup \{\varphi\} \subseteq For(\Sigma)$: 
$\Gamma \vdash_{\lfium_\cons} \varphi$ \ iff \  $\Gamma\models_{\lfium_\cons}^2 \varphi$.
\end{theorem}

\begin{definition}   \label{val-bival-LFI1}
Let $\mu$ be a bivaluation for $\lfium_\cons$. The valuation over  the Nmatrix $\mathcal{M}\big(\mathcal{B}_{\mathbb{A}_2}^{\lfium_\cons}\big)$ induced by $\mu$ is defined as follows: $v_\mu^{\lfium_\cons}(\alpha) \defin (\mu(\alpha),\mu(\neg\alpha))$, for every formula $\alpha$.
\end{definition}

\

Observe that, by  Remark~\ref{pairs}, the snapshots are pairs instead of triples.
Hence, in difference to $v_\mu^\mbc$ (see Definition~\ref{val-bival-mbC}), a third coordinate for $v_\mu^{\lfium_\cons}$ is  not necessary.

\begin{proposition}
Let $\mu$ be a bivaluation for $\lfium_\cons$. Then $v_\mu^{\lfium_\cons}$ is a valuation over $\mathcal{M}\big(\mathcal{B}_{\mathbb{A}_2}^{\lfium_\cons}\big)$ such that $v_\mu^{\lfium_\cons}(\alpha) \in D$ iff $\mu(\alpha)=1$, for every formula $\alpha$.
\end{proposition}
\begin{proof}
It is immediate from Definition~\ref{defKlfi1}, Definition~\ref{bivalLFI1} and the definition of the operations in the Boolean algebra $\mathbb{A}_2$, by observing that $\mu(\cons\alpha)={\sim}(\mu(\alpha) \land \mu(\neg\alpha))$ and $\mu(\neg\cons\alpha)=\mu(\alpha) \land \mu(\neg\alpha)$ (see~\cite{CC16}). The details are left to the reader.
\end{proof}

\

\begin{theorem}  [Adequacy of $\lfium_\cons$ w.r.t. $\mathcal{M}\big(\mathcal{B}_{\mathbb{A}_2}^{\lfium_\cons}\big)$]  \label{adeqLFIumA2}
Let $\Gamma \cup \{\varphi\}$   be a set of formulas in $For(\Sigma)$. Then:
$\Gamma \vdash_{\lfium_\cons} \varphi$ \ iff \  $\Gamma\models_{\mathcal{M}(\mathcal{B}_{\mathbb{A}_2}^{\lfium_\cons})} \varphi$.
\end{theorem}
\begin{proof}
`Only if' part (Soundness): It is an immediate consequence of Theorem~\ref{adeqLFIum}, given that $\mathcal{M}\big(\mathcal{B}_{\mathbb{A}_2}^{\lfium_\cons}\big) \in Mat(\Klfium)$.\\[2mm]
`If' part (Completeness): Suppose that  $\Gamma\models_{\mathcal{M}(\mathcal{B}_{\mathbb{A}_2}^{\lfium_\cons})} \varphi$, and let $\mu$ be a bivaluation for $\lfium_\cons$ such that $\mu[\Gamma] \subseteq \{1\}$. Then $v_\mu^{\lfium_\cons}$ is a valuation over $\mathcal{M}\big(\mathcal{B}_{\mathbb{A}_2}^{\lfium_\cons}\big)$ such that $v_\mu^{\lfium_\cons}[\Gamma] \subseteq D$. By hypothesis, $v_\mu^{\lfium_\cons}(\varphi) = D$ and so $\mu(\varphi)=1$. This means that $\Gamma\models_{\lfium_\cons}^2 \varphi$, therefore  $\Gamma\vdash_{\lfium_\cons} \varphi$ by  Theorem~\ref{comp-bival-LFI1}.
\end{proof}

\

The latter result constitutes a new proof, from the perspective of swap structures, of the adequacy of $\lfium_\cons$ w.t.r. its 3-valued characteristic matrix. It shows that the standard matrix semantics for \dacdot\ (presented as \lfium) can be recovered by means of swap structures semantics. The swap structures for \lfium/\dacdot, seeing as algebras, are nothing else that twist structures. Moreover, this class of algebras is generated by the 3-valued characteristic matrix of \dacdot, as a consequence of Theorem~\ref{teo-repr-Klfium}  below. Thus, the class of algebraic models of  \dacdot\ (in the sense of Blok and Pigozzi) is recovered as an special case of swap structures semantics, as it will analyzed in Section~\ref{swap=twist}. 

As a first step, recall the dual Kalman's functor $K^*_\mbc:\balg \to \swmbc$ for \mbc\ (see Definition~\ref{KalmanFunc}). Clearly,  it can be modified  to a functor $K^*_{\lfium_\cons}:\balg \to \swlfium$, where \swlfium\ is the full subcategory in \swmbc\ formed by the swap structures for $\lfium_\cons$. As in the case of $K^*_\mbc$, the functor $K^*_{\lfium_\cons}$ preserves arbitrary products and monomorphisms and so a Birkhoff-like representation theorem similar to Theorem~\ref{teo-repr-Kmbc} holds for \Klfium:

\begin{theorem} [Representation Theorem for $\Klfium$] \label{teo-repr-Klfium}  Let $\mathcal{B}$ be a swap \linebreak
structure  for $\lfium_\cons$. Then, there exists a set $I$ and a monomorphism  of algebras $\hat h:\mathcal{B} \to \prod_{i \in I}\mathcal{B}_{\mathbb{A}_2}^{\lfium_\cons}$.
\end{theorem}

\

As it will be clarifed in sections~\ref{Kalman-twist}  and~\ref{swap=twist}, the algebra $\mathcal{B}_{\mathbb{A}_2}^{\lfium_\cons}$, together with its 2-element subalgebra $\{T,F\}$, are the only subdirectly irreducible algebras in the class $\Klfium$ of algebras for  \lfium/\dacdot\ (which is polynomially equivalent to the variety of MV-algebras of order 3). From this, Theorem~\ref{teo-repr-Klfium} is nothing else than the standard Birkhoff's representation theorem for  $\Klfium$.

\

\subsection{From Kalman-Cignoli construction to Fidel-Vakarelov twist structures} \label{Kalman-twist} 

For the reader's convenience,  in this section the notion of twist structures and its relationship with a construction of J.~A. Kalman, as it was shown and reworked by R. Cignoli, will be briefly surveyed.

\begin{definition} \label{defDM-Kleene}
A {\em De Morgan lattice} is an algebra $\mathcal{D}=\langle D,\land,\lor,\neg\rangle$ such that the reduct $\mathcal{D}_{\land,\lor}=\langle D, \land,\lor \rangle$ is a distributive lattice and $\neg$ is an unary operator which is a De Morgan negation, that is: $\neg\neg a = a$ and $\neg(a \vee b)=\neg a \land \neg b$ for every $a,b$ (hence $\neg(a \land b)=\neg a \lor \neg b$ for every $a,b$). If $\mathcal{D}_{\land,\lor}$ is a bounded lattice with  bottom and top elements $0$ and $1$, respectively, then  $\mathcal{D}=\langle D,\land,\lor,\neg,0,1\rangle$ is called a {\em De Morgan algebra}. A De Morgan algebra satisfying $a \land \neg a \leq b \lor \neg b$ for every $a,b$ is called a {\em Kleene algebra}. A Kleene algebra is said to be {\em centered} if it has an element $c$ (called a {\em center}) such that $\neg c = c$  (it follows that, if a Kleene algebra has a center, it is unique). 
\end{definition}

In 1958  J.~A. Kalman~\cite{kal:58} shown that, for every bounded distributive lattice $\mathcal{L}=\langle L,\land,\lor,0,1\rangle$ the set $K(\mathcal{L})=\{(a,b) \in L^2 \ : \ a \land b =0\}$ is a centered Kleene algebra where the operations are defined as follows:
$$(a,b) \land (c,d) = (a \land c, b \vee d)$$
$$(a,b) \lor (c,d) = (a \lor c, b \land d)$$
$$\neg (a,b) = (b, a).$$

\noindent The center of $K(\mathcal{L})$ is $(0,0)$. In 1986 R. Cignoli~\cite{cin:86} extended Kalman's construction to a functor as follows: $K(f)(a,b)=(f(a),f(b))$ for every lattice homomorphism $f:\mathcal{L} \to \mathcal{L}'$ and every $(a,b) \in K(\mathcal{L})$. Moreover, among other results, he proves that the functor $K$ has a left adjoint.

\begin{definition} \label{defNelson}
A {\em quasi-Nelson algebra} is a Kleene algebra $\mathcal{N}$  such that for every $a,b$  there exists the relative pseudocomplement $a \Rightarrow (\neg a \vee b)$, which it will be denoted by $a \to b$. That is, $a \to b \defin Max\{c \ : \  a \land c \leq \neg a \vee b \}$. A {\em Nelson algebra} is a quasi-Nelson algebra such that $(a \land b) \to c = a \to (b \to c)$ for every $a,b,c$.
\end{definition}

In~\cite{cin:86} Cignoli observes that M. Fidel~\cite{fid:78} and D. Vakarelov~\cite{vaka:77} have shown independently that the Kalman's construction $K(\mathcal{H})$ produces, for a Heyting algebra $\mathcal{H}$, a Nelson algebra in which $(a,b) \to (c,d) \defin (a \to c, a \land d)$ (on the right-hand side of this equation, and as it was done along the paper, the relative pseudocomplement in a Heyting algebra is denoted by $\to$). This construction is what is was  called {\em twist structures}. In~\cite{cin:86} it is obtained the converse of Fidel-Vakarelov result by showing that, for any bounded distributive lattice $\mathcal{L}$, the centered Kleene algebra $K(\mathcal{L})$ is a (centered) Nelson algebra if and only if $\mathcal{L}$  is a Heyting algebra. 
This construction allows us to study Nelson algebras in terms of twist structures over Heyting algebras. It is worth noting that, in 1966 J.~M. Dunn obtained in his PhD thesis~\cite{dun:66} a representation of De Morgan lattices by means of pairs of sets called {\em proposition surrogates} equipped with operations similar to the ones proposed by  Kalman and by Fidel-Vakarelov for  twist structures.

Besides their construction, Fidel-Vakarelov define a matrix semantics over twist structures in order to semantically characterize Nelson's logic. Given a twist structure  $\mathcal{N}$, the set of designated is given by 
$$D_\mathcal{N} = \{(a,b) \in |\mathcal{N}| \ : \ a=1\}.$$

Aferwards, twist structures semantics have been generalized in the literature to several classes of logics, including  modal logics (see, for instance, \cite{odin:09,odin:wan:10,ram:fer:2009}). In all the cases, each twist structure $\mathcal{N}$ have associated a logical matrix $\mathcal{M}(\mathcal{N})=(\mathcal{N},D_\mathcal{N})$ defined as above.

Returning to Kalman's construction, Cignoli have shown in~\cite[Lemma~4.1]{cin:86} that the Kalman's functor $K$, when restricted to Boolean algebras (which are, of course, special cases of Heyting algebras), produces Nelson algebras which are semisimple. On the other hand, A. Monteiro has shown in~\cite{mon:63} that the variety of semisimple Nelson algebras is polynomially equivalent  to the variety of MV-algebras of order 3 (see~\cite[Corollary~5.5]{cin:86}). As it is well-known, the latter is the variety associated to \L ukasiewicz 3-valued logic \luka\ by means of the Blok-Pigozzi algebraization technique. Being so, the Kalman's construction, when restricted to Boolean algebras, produces a twist-structures semantics for \luka. In particular, $K(\mathbb{A}_2)$ produces the 3-valued semisimple Nelson algebra $\mathcal{N}_3= \langle N_3, \land,\lor,\to,\neg,F,T\rangle$ such that $N_3=\{F,f,T\}$ where $F=(0,1)$, $f=(0,0)$ and $T=(1,0)$. The tables for $\land$ and $\lor$ correspond to meet and join lattice operations (assuming that $F \leq f \leq T$), while the De Morgan negation $\neg$
is given by

\begin{center}
\begin{tabular}{|c||c|} \hline
$\quad$ & $\neg$ \\
 \hline \hline
    $T$   & $F$    \\ \hline
     $f$   & $f$    \\ \hline
     $F$   & $T$   \\ \hline
\end{tabular}
\end{center}

\ \\

\noindent 
By definition of twist structures semantics (see above), the set of designated values is given by $D_{\mathcal{N}_3}=\{T\}$. It is worth noting that the  usual implication $\to_J$ of \luka\ can be defined as $x \to_J y \defin (x \to y) \land (\neg y \to \neg x)$.   Thus, it is clear that this twist structures semantics produces, indeed, the usual class of models of \luka.

\subsection{Swap structures meet twist structures} \label{swap=twist}

As it was observed in Section~\ref{secLFI1}, the technique of swap structures allows a twist structures semantics for \lfium/\dacdot. An interesting fact is that this semantics is dual to the twist structures semantics for \L ukasiewicz 3-valued logic \luka\  obtained by R. Cignoli in~\cite[Section~4]{cin:86} by using the Kalman's functor, as described in the previous section. 

Indeed, consider again the dual Kalman's functor $K^*_{\lfium_\cons}:\balg \to \swlfium$ described at the end of Section~\ref{secLFI1}. It is worth noting that the Kalman's functor $K$ --restricted to the category \balg\ of Boolean algebras-- and  $K^*_{\lfium_\cons}$, despite being defined in the same way for morphisms, they do not coincide at the level of objects. However, they produce objects which are dual in the following sense: recalling that $|K^*_{\lfium_\cons}(\mathcal{A})|=\textsc{B}_\mathcal{A}^\axwci$ for every a  Boolean algebra $\mathcal{A}$,  the mapping  $*:K(\mathcal{A}) \to \textsc{B}_\mathcal{A}^\axwci$ given by $*(a,b)= ({\sim}a,{\sim}b)$ is a bijection such that
$*(z \land w) = *z \lor *w$; $*(z \lor w) = *z \land *w$; $*\neg z = \neg {*}z$; $* T = F$; $* f = t$ and $* F = T$.\footnote{In order to simplify the presentation, in these equations we are considering the single-valued full swap structure $\mathcal{B}_{\mathcal{A}}^{\lfium_\cons}$ over $\mathcal{A}$ as an ordinary algebra. Additionally, observe that $T$, $t$, $f$ and $F$ are defined for every Boolean algebra.}

On the other hand, in~\cite[Theorem~4.3]{blok:pig:01} W. Blok and D. Pigozzi have shown that two logic systems which are inter-translatable in a strong sense cannot be distingued from the point of view of algebra, in the sense that if one of the systems is algebraizable then the other will be also algebraizable w.r.t. the same quasi-variety. As an illustrative example, they observe in~\cite[Example~4.1.2]{blok:pig:01} that the class of algebraic models of \dacdot\ is polynomially equivalent to the variety of  MV-algebras of order 3, the class of algebraic models of \luka, given that both logics are inter-translatable in such sense. Being so, the class of swap structures (seen as algebras) for  $\lfium_\cons$ generated by $K^*_{\lfium_\cons}(\mathbb{A}_2)$ in the sense of Theorem~\ref{teo-repr-Klfium} coincides, up to language, with the class of algebras for \dacdot\ generated by $K(\mathbb{A}_2)$.

The relationship between $K$ and  $K^*_{\lfium_\cons}$ pointed out above justifies the name {\em dual Kalman's functors} given to the functors for swap structures introduced here.

\begin{remark}
Reinforcing this argument, recall that Cignoli's construction described at  the end of Section~\ref{Kalman-twist} constitutes an original twist-structure semantics for \luka. In such construction, the 3-valued characteristic matrix of \luka\ can be recovered from $K(\mathbb{A}_2)$ in which there is only one designated element, namely  $D_\mathcal{N}=\{T\}$. Our construction is dual in the sense that the 3-valued characteristic matrix of \lfium/\dacdot\ is recovered instead of that of \luka, hence there are now two designated elements given by the set $D_\mathcal{N}=\{T,t\}$. This confirms, from  a different perspective, that \dacdot\ and \luka\ are dual logics in which the latter is paracomplete (that is, a sentence and its negation can be both false, but never both true at the same time) while the former is  paraconsistent (that is, a sentence and its negation can be both true, but never both false at the same time). In terms of pairs: given $(a,b) \in  K(\mathcal{A})$, $a \land b=0$ but not necessarily $a \lor b=1$. On the other hand, if $(a,b) \in \textsc{B}_\mathcal{A}^\axwci$ then $a \lor b=1$ but it is not always the case that $a \land b=0$. 
\end{remark}

\

The logic $\lfium_\cons$ is obtained from \mbc, \mbcciw\ and the other \lfis\ studied here by adding enough axioms to such logics. The weaker systems are characterized by non-deterministic swap structures, while $\lfium_\cons$, because of the logical power of the additional axioms, produces deterministic swap structures, which can be identified with twist structures. Looking from the opposite perspective, it could be argued that swap structures in general can be seen as non-deterministic twist structures: for instance, the swap structures semantics obtained for \mbc, \mbcciw, \mbcci\ and \ci\ in the previous sections could be considered as a non-deterministic twist structures semantics for them.  Moreover, the fact that the Kalman-Cignoli functor can be generalized to the wider non-deterministic context of swap structures provides additional support for this claim. 

Clearly, the wider approach given by swap structures has several disvantages with respect to the more traditional approach given by twist structures. On the one hand, the latter is based on ordinary algebras, thus all the machinery of universal algebra can be used. On the other hand, swap structures are based on non-deterministic algebras and such structures, as it was briefly discussed in  Section~\ref{intro}, does not offer a uniform and well-established formal treatment as a generalized class of algebras: each notion from ordinary algebra admits several generalizations to the non-deterministic framework. Being so, it could be not expected that the dual Kalman's functors 
$K^*_{\mathbf{L}}$ for a given logic {\bf L} has a left adjoint as in the case of the Kalman's functor. The existence of such left adjoint for each logic {\bf L} is an interesting topic of further research.

\

\subsection{Swap/twist structures semantics for \ciore} \label{sectCiore}

Finally, the same techiques will be applied to obtain a twist structures semantics for a 3-valued \lfi\ called \ciore, as a particular (or limiting) case of swap structures. This will give additional support  to the idea that swap structures corresponds to non-deterministic twist structures, since when applied to algebraizable logics characterized by twist structures they produce exactly the algebras associated to it through the twist structures.

The main feature of \ciore\ is that it presents a strong property of propaga\-tion/retro-propagation of consistency. Thus, $\cons\alpha$ is implied by $\cons p$, for any propositional $p$ occurring in $\alpha$. In formal terms: 

\begin{definition} \label{ciore} (\cite[Definition~4.3.9]{CC16}) Let \ciore\ be the logic over $\Sigma$ obtained from \ci\   (see Definition~\ref{mbcci-eCPL}(2))  by adding the following axiom schemas:
$$\begin{array}{ll}
\alpha \imp \neg\neg\alpha &    \hspace{2cm}        (\axce)\\[2mm]
(\cons\alpha \lor \cons\beta) \leftrightarrow \cons(\alpha\land \beta) & \hspace{2cm} (\axco_1)\\[2mm]
(\cons\alpha \lor \cons\beta) \leftrightarrow \cons(\alpha\lor \beta) & \hspace{2cm} (\axco_2)\\[2mm]
   (\cons\alpha \lor \cons\beta) \leftrightarrow \cons(\alpha\imp \beta) & \hspace{2cm} (\axco_3)\\[2mm]
\end{array}$$
\end{definition}

\begin{remark}
It can be proven that $\cons\alpha \leftrightarrow \cons\neg\alpha$ is derivable in \ciore, for every $\alpha$. From this,  and as mentioned above, $\cons p \rightarrow \cons\alpha$ is derivable in \ciore, for any propositional variable $p$ occurring in $\alpha$.  As   $\lfium_\cons$, the logic \ciore\ is algebraizable in the sense of Blok and Pigozzi (see~\cite[Theorem~4.3.18]{CC16}).
\end{remark}

The semantics of \ciore\ is given by a 3-valued logical matrix which constitutes a slight variation of the corresponding for $\lfium_\cons$. It is defined over the domain  $\textsc{B}_{\mathbb{A}_2}^\axwci = \big\{T, \, t, \,  F\big\}$ such that  $\textrm{D}_3=\{T, \, t\}$ is the set of designated values, and the operations are defined as follows:

\ 

\begin{center}
\begin{tabular}{|c|c|c|c|}
\hline
 $\land$ & $T$   & $t$ & $F$ \\
 \hline \hline
    $T$    & $T$   & $T$   & $F$   \\ \hline
     $t$    & $T$   & $t$   & $F$  \\ \hline
     $F$    & $F$  & $F$ & $F$  \\ \hline
\end{tabular}
\hspace{1cm}
\begin{tabular}{|c|c|c|c|}
\hline
 $\lor$ & $T$   & $t$  & $F$  \\
 \hline \hline
    $T$    & $T$   & $T$  & $T$    \\ \hline
     $t$    & $T$   & $t$  & $T$   \\ \hline
     $F$    & $T$   & $T$  & $F$   \\ \hline
\end{tabular}
\end{center}

\

\begin{center}
\begin{tabular}{|c|c|c|c|}
\hline
 $\imp$ & $T$   & $t$ & $F$ \\
 \hline \hline
    $T$    & $T$   & $T$  & $F$   \\ \hline
     $t$    & $T$   & $t$  & $F$  \\ \hline
     $F$    & $T$   & $T$  & $T$   \\ \hline
\end{tabular}
\hspace{1cm}
\begin{tabular}{|c||c|} \hline
$\quad$ & $\neg$ \\
 \hline \hline
    $T$   & $F$    \\ \hline
     $t$   & $t$     \\ \hline
     $F$   & $T$     \\ \hline
\end{tabular}
\hspace{1cm}
\begin{tabular}{|c||c|}
\hline
 $\quad$ & $\cons$ \\
 \hline \hline
    $T$   & $T$     \\ \hline
     $t$   & $F$    \\ \hline
     $F$   & $T$     \\ \hline
\end{tabular}
\end{center}

\ \\

Consider now the swap structures for \ciore. By means of an analysis similar to that for \lfium/\dacdot, it will be shown tat the swap structures for \ciore\ are, indeed, twist structures given by single-valued operations.

Thus, fix $\#\in \{\wedge,\vee, \to\}$. Let $\mathcal{B} \in \Kci$ and let $z,w \ \in |\mathcal{B}|$. If $u \in z \# w$ then, by Definition~\ref{Kci}(2), $u_1=z_1 \# w_1$. Hence, $\pi_1[\cons(z \# w)]=\{{\sim}((z_1 \# w_1) \land u_2) \ : \ u \in z_1 \# w_1\}$. On the other hand, $\pi_1[\cons z \vee \cons w]=\{{\sim}(z_1 \land z_2) \vee {\sim}(w_1 \land w_2)\}$. By axioms (\axco$_1$)-(\axco$_3$) both sets coincide and so
${\sim}((z_1 \# w_1) \land u_2) =  {\sim}(z_1 \land z_2) \vee {\sim}(w_1 \land w_2)$, that is, $(z_1 \# w_1) \land u_2 =  (z_1 \land z_2) \wedge (w_1 \land w_2)$
for every $u \in z \# w$. This produces a system of two equations on the variable $u_2$ in the Boolean algebra~$\mathcal{A}$:

$$\begin{array}{lll}
a \land u_2 & = & b\\[2mm]
a \lor u_2 & = & 1\\[2mm]
\end{array}$$
where $a=z_1 \# w_1$ and $b= (z_1 \land z_2) \wedge (w_1 \land w_2)$.
It is easy to see that $b \leq a$ for every $\#$, thus  there is just one solution to these equations given by $u_2=(z_1 \# w_1) \to ((z_1 \land z_2) \wedge (w_1 \land w_2))$. Since the negation and the consistency operator behave as in $\lfium_\cons$, this leads us to the following definition:

\begin{definition}  \label{defKCiore}
A {\em swap structure for \ciore} is any $\mathcal{B} \in \Kci$ such that the multioperations are single-valued and defined as follows, for every $(z_1,z_2),(w_1,w_2) \in |\mathcal{B}|$:

\begin{itemize}
 \item[(i)] $(z_1,z_2)\# (w_1,w_2) = \{(z_1\#w_1,(z_1 \# w_1) \to ((z_1 \land z_2) \wedge (w_1 \land w_2)))\}$, for each $\#\in \{\wedge,\vee, \to\}$;
\item[(ii)] $\neg (z_1,z_2) = \{(z_2,z_1)\}$;
 \item[(iii)] $\circ (z_1,z_2) = \{({\sim}(z_1 \wedge z_2),z_1 \wedge z_2)\}$.\\[-2mm]
\end{itemize}
The class of swap structures for \ciore\ will be denoted by \Kciore.
\end{definition}

\begin{definition} \label{maxCiore}
Given a Boolean algebra $\mathcal{A}$, the {\em full swap structure  for \ciore\ over $\mathcal{A}$}, denoted by $\mathcal{B}_\mathcal{A}^{\ciore}$, is the unique swap structure  for \ciore\ defined over $\mathcal{A}$ with domain $\textsc{B}_\mathcal{A}^\axwci$.
\end{definition}

\begin{proposition} \label{CarAxKciore} Let $\mathbf{Ax}'$ be the set of axioms added to \ci\ in order to obtain \ciore\ (recall Definition~\ref{ciore}). Then:
$$\begin{array}{lll}
\Kciore & = & \{ \mathcal{B} \in \Kci \ : \ \ 
\mbox{$\mathcal{M}(\mathcal{B})$ validates all the axioms in $\mathbf{Ax}'$} \}.
\end{array}$$
\end{proposition}
\begin{proof} 
{\bf Part 1:} If  $\mathcal{B} \in \Kciore$ then $\mathcal{B} \in \Kci$ such that $\mathcal{M}(\mathcal{B})$ validates all the axioms in $\mathbf{Ax}'$.\\
Let $\mathcal{B} \in \Kciore$, and let $v$ be a valuation over $\mathcal{B}$.
Let  $\gamma=\alpha \imp \neg\neg\alpha$ be an instance of axiom (\axce). As in the proof of Proposition~\ref{CarAxKlfium}, it can be seen that $v(\gamma) \in D_\mathcal{B}$.

Now, let $\gamma'= (\cons\alpha \lor \cons\beta) \leftrightarrow \cons(\alpha\# \beta)$ be an instance of an axiom in (\axco$_1$)-(\axco$_3$). Let $z=v(\alpha)$ and $w=v(\beta)$. Hence, $v(\alpha \# \beta)=(z_1\#w_1,u_2)$ for $u_2=(z_1 \# w_1) \to ((z_1 \land z_2) \wedge (w_1 \land w_2))$.  Observe that$(z_1 \# w_1) \land u_2=(z_1 \land z_2) \wedge (w_1 \land w_2)$ by the analysis before Definition~\ref{defKCiore}. Then, by definition of $\cons$, $\pi_1(v(\cons(\alpha \# \beta)))={\sim}((z_1 \land z_2) \wedge (w_1 \land w_2))$. On the other hand $\pi_1(v(\cons\alpha))={\sim}(z_1 \wedge z_2)$ and $\pi_1(v(\cons\beta))={\sim}(w_1 \wedge w_2)$, and so $\pi_1(v(\cons \alpha \vee \cons\beta)) = {\sim}(z_1 \wedge z_2) \vee {\sim}(w_1 \wedge w_2) = \pi_1(v(\cons(\alpha \# \beta)))$. Thus, $\pi_1(v(\gamma'))=1$ for every instance $\gamma'$ of any axiom in (\axco$_1$)-(\axco$_3$).\\[2mm]
{\bf Part 2:} If  $\mathcal{B} \in \Kci$ such that $\mathcal{M}(\mathcal{B})$ validates all the axioms in $\mathbf{Ax}'$ then $\mathcal{B} \in \Kciore$.\\
Fix $\mathcal{B} \in \Kci$ such that $\mathcal{M}(\mathcal{B})$ validates all the axioms in $\mathbf{Ax}'$. Let $z \in |\mathcal{B}|$.  As in the proof of Proposition~\ref{CarAxKlfium} it can be seen that $\neg (z_1,z_2) = \{(z_2,z_1)\}$. 

With respect to the binary multioperators, fix $\#\in \{\wedge,\vee, \to\}$ and  let $z,w,u \ \in |\mathcal{B}|$ such that $u \in z \# w$.  By Definition~\ref{Kci}(2) it follows that $u_1=z_1 \# w_1$. Consider two different propositional variables $p,q$ and a valuation $v$ over $\mathcal{B}$ such that $v(p)=z$, $v(q)=w$ and $v(p \# q)=u$. Then $\pi_1(v(\cons p \vee \cons q))= \pi_1(v(\cons p)) \vee \pi_1(v(\cons q)) = {\sim}(z_1 \wedge z_2) \vee {\sim}(w_1 \wedge w_2)$. On the other hand, $\pi_1(v(\cons(p \# q))) = {\sim}((z_1 \# w_1) \wedge u_2)$. By axioms (\axco$_1$)-(\axco$_3$), $\pi_1(v(\cons p \vee \cons q))=\pi_1(v(\cons(p \# q)))$ and so, by applying $\sim$ to both sides of the last equation, $(z_1 \land z_2) \wedge (w_1 \land w_2)=(z_1 \# w_1) \land u_2$. Given that $(z_1 \# w_1) \vee u_2=1$ (since $u \in \textsc{B}_{\mathcal{A}}^\axwci$) it follows that
$u_2=(z_1 \# w_1) \to ((z_1 \land z_2) \wedge (w_1 \land w_2)))$, by the analysis done 
 before Definition~\ref{defKCiore}. Therefore each binary multioperation $\#$ in $\mathcal{B}$ is single-valued, and it is defined as in Definition~\ref{defKCiore}. That is, $\mathcal{B} \in \Kciore$.
\end{proof}

The following result can be proven by easily adapting the proof of Theorem~\ref{adeqLFIum}:

\begin{theorem}  [Adequacy of \ciore\ w.r.t. swap structures]  \label{adeqCiore}
Let $\Gamma \cup \{\varphi\}$  be a set of formulas  in $For(\Sigma)$. Then:
$\Gamma \vdash_{\ciore} \varphi$ \ iff \  $\Gamma\models_{Mat(\Kciore)} \varphi$.
\end{theorem}

The logic \ciore\ can be characterized  in terms of the 3-valued Nmatrix defined over $\mathbb{A}_2$.  This corresponds to the adequacy of \ciore\ w.r.t. its 3-valued standard logical matrix, see~\cite[Theorem~4.4.29]{CC16}. Thus, consider the following notion of bivaluations for \ciore:

\begin{definition} [\cite{CC16}] \label{bivalciore}
A {\em bivaluation for \ciore} is a bivaluation $\mu:For(\Sigma)\to \big\{0,1\big\}$ for \mbc\ (recall Definition~\ref{bivalold}) which  satisfies, in addition, the   following clauses:\\[2mm]
{\bf (\valCi)} \ $\mu(\neg\cons \alpha) = 1$ \ implies \ $\mu(\alpha)=\mu(\neg \alpha)=1$ \\[2mm]
{\bf (\valcef)} \ $\mu(\lnot\lnot \alpha)=1$ \ iff  \ $\mu(\alpha)=1$ \\[2mm]
{\bf (\valco$_1$)} \ $\mu(\cons\alpha)=1$ \ or \ $\mu(\cons \beta)=1$ \ iff \  $\mu(\cons(\alpha\land \beta))=1$.\\[2mm]
{\bf (\valco$_2$)} \ $\mu(\cons\alpha)=1$ \ or \ $\mu(\cons \beta)=1$ \ iff \  $\mu(\cons(\alpha\lor \beta))=1$.\\[2mm]
{\bf (\valco$_3$)} \ $\mu(\cons\alpha)=1$ \ or \ $\mu(\cons \beta)=1$ \ iff \  $\mu(\cons(\alpha\imp \beta))=1$.
\end{definition}

\

The proof of the following result is analogous to that for $\lfium_\cons$:

\begin{theorem}  [Adequacy of \ciore\ w.r.t. $\mathcal{M}\big(\mathcal{B}_{\mathbb{A}_2}^{\ciore}\big)$]  \label{adeqCioreA2}
Let $\Gamma \cup \{\varphi\}$  be a set of formulas in $For(\Sigma)$. Then:
$\Gamma \vdash_{\ciore} \varphi$ \ iff \  $\Gamma\models_{\mathcal{M}(\mathcal{B}_{\mathbb{A}_2}^{\ciore})} \varphi$.
\end{theorem}

\

Finally, and as in the previous cases, a Birkhoff-like decomposition theorem   can be obtained for swaps structures for \ciore. Indeed,  the dual Kalman's functor $K^*_\mbc:\balg \to \swmbc$ for \mbc\ (see Definition~\ref{KalmanFunc}) can be easily modified  to a functor $K^*_{\ciore}:\balg \to \swciore$, where \swciore\ is the full subcategory in \swmbc\ formed by the swap structures for \ciore. By adapting the proof for  \mbc\ it can be seen that the functor $K^*_{\ciore}$ preserves arbitrary products and monomorphisms and so the following holds:

\begin{theorem} [Representation Theorem for $\Kciore$] \label{teo-repr-Kciore}  Let $\mathcal{B}$ be a swap \linebreak
structure  for \ciore. Then, there exists a set $I$ and a monomorphism  of algebras $\hat h:\mathcal{B} \to \prod_{i \in I}\mathcal{B}_{\mathbb{A}_2}^{\ciore}$.
\end{theorem}

\

Different from the case of  $\lfium_\cons$, it could not be asserted that the latter result is an ordinary Birkhoff's representation theorem. Indeed,  despite the structures of $\Kciore$ being ordinary algebras, it is not immediate to see that the algebra $\mathcal{B}_{\mathbb{A}_2}^{\ciore}$ is subdirectly irreducible in the class $\Kciore$ of \ciore-algebras. A formal study of the class $\Kciore$ deserves future research.

\

\section{Concluding remarks and future work} \label{conclusion}

This paper proposes the use of multialgebras as a valid alternative to the standard techniques from algebraic logics, apt to deal with logics which lie outside the scope of such techniques. 
Specifically, the class of multialgebras known as swap structures are studied from the point of view of universal algebra, by adapting standard concepts to multialgebras in a suitable way. This allows to analyze  categories of swap structures for several logics of formal inconsistency (\lfis), obtaining so a representation theorem for each class of swap structures  which resembles  the well-known Birkhoff's representation theorem for algebras. 

In the case of the algebraizable 3-valued logic \dacdot\ (which is dual to  \L ukasewicz 3-valued logic \luka) studied in Section~\ref{sectwist-swap}, our representation theorem coincides with the original Birkhoff's representation theorem. In addition, the swap structures became twist structures in the sense of Fidel~\cite{fid:78} and Vakarelov~\cite{vaka:77}.
Moreover,  the dual Kalman's functor for swap structures can be seen as a generalization of the original construction of Kalman applied to 3-valued logics.
This gives us support to argue that the swap structures semantics (which are non-deterministic algebras), together with the associated dual Kalman's functor, would corresponds to non-deterministic twist structures, able to give a multialgebraic counterpart to non-algebraizable logics. 

However, there are many questions to be answered. The original Kalman's functor (and the associated twist-structures semantics) allows to represent classes of algebras in terms of pairs of elements over other classes of algebras. For instance, Nelson algebras can be represented by means of pairs of elements in a Heyting algebra. It is fundamental to observe that the output of the Kalman functor can be abstracted to an axiomatized class of algebras. Thus, the output of the Kalman's functor applied to the class of Heyting can be abstracted by means of  the class of Nelson algebras. In general, it is an important issue to axiomatize a given class of twist structures in order to represent it as a class of standard algebras (see, for instance, \cite{ono:rivi:14,riv:2014}). One of the main topics of future research in the present framework is how to axiomatize given classes of swap structures, as it is done for twist structures. This leads us to the theory of varieties and quasi-varieties of multialgebras. More generally, the development of a equation theory in the framework of multialgebras  suitable to deal with such structures deserves future research.

The study of a theory of identities in multialgebras is also related to another important question to be investigated, namely the Birkhoff's representation theorem for multialgebras (and, in particular, for swap structures). The representation theorems given for \Kmbc\ and the other classes of swap structures can be seen as a generalized form of Birkhoff's representation theorem. As mentioned in Remark~\ref{rem-subdir-irred}, an open question is to characterize the notion of subdirectly irreducible multialgebras, which would lead to a satisfactory generalization of Birkhoff's theorem for multialgebras. 
Some results related with Birkhoff's theorem for multialgebras were already proposed in the literature, but the problem is far to be absolutely solved. For instance, G. Hansoul propose in~\cite{han:83} a version of Birkhoff's representation theorem only for finitary multialgebras,  that is, multialgebras in which the multioperations produce finite sets of possible-values for a given entry. On the other hand, D. Schweigert~\cite{sch:85} only sketches a possible proof of Birkhoff's theorem without specifying the basic definitions from  the theory of multialgebras being adopted.

It is worth mentioning that  X. Caicedo obtains  in  \cite{XC} a satisfactory generalization of Birkhoff's representation theorem for first-order structures. However, the application of Caicedo's result to multialgebras is not immediate, despite multialgebras being particular cases of first-order structures. The problem arises because of the tigh notions of homomorphisms and subalgebras coming from Model Theory, which are not compatible with the weaker ones  adopted here in the context of multialgebras. This is why obtaining a Birkhoff's representation theorem for swap structures (or, in general, for multialgebras) remains an important open problem.

To conclude, we consider that the use of multialgebras, and swap structures  in particular,  can expands the horizons of the traditional approach to algebraization of logics. Moreover, the study of multialgebras (and first-order structures in general) from the perspective of universal algebra is a  topic that deserves further research. 

\

\section*{Acknowledgments}
This paper is a revised and extended version of the preprint~\cite{CFG:2016}.
Coniglio was financially supported by an individual research
grant from CNPq, Brazil (308524/2014-4). Figallo-Orellano acknowledges support from a post-doctoral grant  from FAPESP, Brazil (2016/21928-0). Golzio was financially supported by a PhD grant  from FAPESP, Brazil  (2013/04568-1).

\

\end{document}